\documentclass[twoside,a4paper,10pt]{amsart}

\usepackage{amscd}
\usepackage{amsmath}
\usepackage{amssymb}
\usepackage{amsthm}
\usepackage{color}
\usepackage{dutchcal}
\usepackage{fullpage}
\usepackage{graphicx}
\usepackage{hyperref}
\usepackage[noabbrev,capitalize]{cleveref}
\usepackage{mathrsfs}
\usepackage{mathtools}
\usepackage{nccmath}
\usepackage{pict2e}
\usepackage{stackrel}
\usepackage[T1]{fontenc}
\usepackage{tikz}
\usetikzlibrary{decorations.pathmorphing}
\usepackage[utf8]{inputenc}
\usepackage{wasysym}
\usepackage[all]{xy}
\usepackage{xcolor}
\usepackage{xy}

\hypersetup{
    colorlinks,
    linkcolor={red!50!black},
    citecolor={blue!50!black},
    urlcolor={blue!80!black}
}

\usepackage{palatino}
\allowdisplaybreaks

\newtheorem{lemma}{Lemma}[section]
\newtheorem{remark}[lemma]{Remark}
\newtheorem{example}[lemma]{Example}
\newtheorem{theorem}[lemma]{Theorem}
\newtheorem{corollary}[lemma]{Corollary}
\newtheorem{definition}[lemma]{Definition}
\newtheorem{proposition}[lemma]{Proposition}

\newcommand{\A}{\ensuremath{\ensuremath{\mathcal{s}\mathscr{A}}}}
\newcommand{\as}{\ensuremath{\mathrm{As}}}
\newcommand{\com}{\ensuremath{\mathrm{Com}}}
\newcommand{\C}{\ensuremath{\mathscr{C}}}
\renewcommand{\L}{\ensuremath{\ensuremath{\mathcal{s}\mathscr{L}}}}
\newcommand{\lie}{\ensuremath{\mathrm{Lie}}}
\renewcommand{\P}{\ensuremath{\mathscr{P}}}
\newcommand{\Q}{\ensuremath{\mathscr{Q}}}
\newcommand{\susp}{\ensuremath{\mathscr{S}}}

\newcommand{\antishriek}{\mbox{\footnotesize{\rotatebox[origin=c]{180}{$!$}}}}
\renewcommand{\bar}{\ensuremath{\mathrm{B}}}
\newcommand{\Ccog}{\ensuremath{\mathscr{C}\text{-}\mathsf{cog}}}

\newcommand{\cocom}{\mathsf{coCom}}
\newcommand{\F}{\mathcal{F}}
\newcommand{\FLalg}{\ensuremath{\infty\text{-}\mathcal{s}\mathscr{L}_\infty\text{-}\mathsf{alg}^{\mathcal{F}}}}
\newcommand{\g}{\mathfrak{g}}
\newcommand{\h}{\mathfrak{h}}

\newcommand{\id}{\operatorname{id}}
\renewcommand{\k}{\ensuremath{\mathbb{K}}}
\newcommand{\q}{\ensuremath{\mathbb{Q}}}

\newcommand{\Lalg}{\ensuremath{\mathcal{s}\mathscr{L}_\infty\text{-}\mathsf{alg}}}
\newcommand{\M}{\rotatebox[origin=c]{180}{$\mathsf{M}$}}

\newcommand{\map}{\ensuremath{\mathrm{Map}_*}}
\newcommand{\MC}{\ensuremath{\mathrm{MC}}}
\newcommand{\op}{\ensuremath{\mathsf{Op}}}
\newcommand{\Palg}{\ensuremath{\mathscr{P}\text{-}\mathsf{alg}}}
\newcommand{\proj}{\ensuremath{\mathrm{proj}}}
\newcommand{\quil}{\widetilde{\mathcal{C}}\lambda}
\renewcommand{\S}{\ensuremath{\mathbb{S}}}
\newcommand{\Sh}{\ensuremath{\mathrm{Sh}}}
\newcommand{\sh}{\ensuremath{\operatorname{sh}}}
\newcommand{\ssets}{\ensuremath{\mathsf{sSets}}}
\newcommand{\Top}{\mathsf{Top}}
\newcommand{\Tw}{\ensuremath{\mathrm{Tw}}}

\makeatletter
\newcommand{\adjunction}{\@ifstar\named@adjunction\normal@adjunction}
\newcommand{\normal@adjunction}[4]{%
  #1\colon #2%
  \mathrel{\vcenter{%
    \offinterlineskip\m@th
    \ialign{%
      \hfil$##$\hfil\cr
      \longrightharpoonup\cr
      \noalign{\kern-.3ex}
      \smallbot\cr
      \longleftharpoondown\cr
    }%
  }}%
  #3 \noloc #4%
}
\newcommand{\named@adjunction}[4]{%
  #2%
  \mathrel{\vcenter{%
    \offinterlineskip\m@th
    \ialign{%
      \hfil$##$\hfil\cr
      \scriptstyle#1\cr
      \noalign{\kern.1ex}
      \longrightharpoonup\cr
      \noalign{\kern-.3ex}
      \smallbot\cr
      \longleftharpoondown\cr
      \scriptstyle#4\cr
    }%
  }}%
  #3%
}
\newcommand{\longrightharpoonup}{\relbar\joinrel\rightharpoonup}
\newcommand{\longleftharpoondown}{\leftharpoondown\joinrel\relbar}
\newcommand\noloc{%
  \nobreak
  \mspace{6mu plus 1mu}
  {:}
  \nonscript\mkern-\thinmuskip
  \mathpunct{}
  \mspace{2mu}
}
\newcommand{\smallbot}{%
  \begingroup\setlength\unitlength{.15em}%
  \begin{picture}(1,1)
  \roundcap
  \polyline(0,0)(1,0)
  \polyline(0.5,0)(0.5,1)
  \end{picture}%
  \endgroup
}
\makeatother

\author{Daniel Robert-Nicoud \and Felix Wierstra}

\date{}
\title{Homotopy morphisms between convolution homotopy Lie algebras}
\address{Daniel Robert-Nicoud, UBS Business Solutions AG, Max-H\"onnger-Strasse 80, 8048 Z\"urich}
\email{\href{mailto:daniel.robertnicoud@gmail.com}{daniel.robertnicoud@gmail.com}}
\address{Felix Wierstra, LAGA, CNRS, UMR 7539, Universit\'e Sorbonne Paris Nord, Sorbonne Paris Cit\'e, Universit\'e Paris 8, 99 Avenue Jean-Baptiste Cl\'ement, 93430 Villetaneuse, France \newline Felix Wierstra, Stockholm University, Department of Mathematics, Stockholm University, SE - 106 91 Stockholm, Sweden}
\email{\href{mailto:felix.wierstra@gmail.com}{felix.wierstra@gmail.com}}

\begin{document}
	
	\begin{abstract}
		In previous works by the authors --- \cite{w16}, \cite{rn17} --- a bifunctor was associated to any operadic twisting morphism, taking a coalgebra over a cooperad and an algebra over an operad, and giving back the space of (graded) linear maps between them endowed with a homotopy Lie algebra structure. We build on this result by using a more general notion of $\infty$-morphism between (co)algebras over a (co)operad associated to a twisting morphism, and show that this bifunctor can be extended to take such $\infty$-morphisms in either one of its two slots. We also provide a counterexample proving that it cannot be coherently extended to accept $\infty$-morphisms in both slots simultaneously. We apply this theory to rational models for mapping spaces.
	\end{abstract}
	
	\maketitle

\setcounter{tocdepth}{1}

\tableofcontents

\section{Introduction}

Homotopy Lie algebras --- hereafter called $\mathscr{L}_\infty$-algebras --- and their non-sym\-metric counterparts, $\mathscr{A}_\infty$-algebras, are structures which are ubiquitous in mathematics and theoretical physics. For example, they appear naturally in Kontsevich's proof of deformation quantization of Poisson manifolds \cite{kon03}, in string field theory \cite{zwie93}, in derived deformation theory \cite{prid10}, \cite{ss12} and others, as algebras of symmetries for conformal field theories \cite{bft17}, in symplectic topology \cite{kontsevich95}, \cite{seidel08}, \cite{fukayaohohtaono} and others, as rational models for mapping spaces \cite{ber15}, and in many other places. The good notion of morphism for these structures that one needs in order to study their homotopy theory is that of  $\infty$-morphisms, i.e. a generalized notion of morphisms of e.g. $\mathscr{L}_\infty$-algebras, encoding maps that are morphisms of algebras only "up to coherent homotopies". These $\infty$-morphisms also appear in the mathematical and physical landscape, for example again as a fundamental ingredient in the proof of deformation quantization of Poisson manifolds by Kontsevich.

\medskip

The goal of the present paper is to show how one can functorially produce $\mathscr{L}_\infty$-algebras and $\infty$-morphisms between them starting from couples formed by a coalgebra and an algebra respectively over a cooperad and an operad related in a certain way. More precisely, a notion of (co)algebra of a certain type is encoded by a (co)operad. We fix a cooperad $\C$ and an operad $\P$. A twisting morphism
\[
\alpha:\C\longrightarrow\P
\]
is a map satisfying the Maurer--Cartan equation in a certain associated Lie algebra. In their previous articles \cite{rn17} and \cite{w16}, the authors showed that the data of such a twisting morphism is equivalent to a natural $\mathscr{L}_\infty$-algebra structure on the chain complex of linear maps between a $\C$-coalgebra and a $\P$-algebra. This result is recalled in \cref{thm:bijTwHoms}. This construction is functorial with respect to morphisms of $\C$-coalgebras in the first slot, and morphisms of $\P$-algebras in the second one.

\medskip

One can relax the notion of morphisms of (co)algebras and use the twisting morphism $\alpha$ to define $\infty$-morphisms of $\C$-coalgebras and of $\P$-algebras. Versions of this generalized notion of $\infty$-morphism have already implicitly appeared in \cite{Mar04} and were explicitly written down in \cite{Ber14}. It is natural to ask whether the above functor is also functorial with respect to these generalized notions of morphisms. We prove that the answer is a partial yes: indeed, one can consider such $\infty$-morphisms in either one of the slots and preserve functoriality, but not in both at the same time. These results are presented in \cref{thm:twoBifunctors} and \cref{sect:counterexample}, and form the bulk of the novel results of this article.

\medskip

We develop this general operadic theory in order to apply it to the study of the rational homotopy theory of mapping spaces in \cref{sect:application}. In that section, we construct a rational model for the based mapping space $\map(K,L_{\q})$ between two simplicial sets $K$ and $L$. This model is given by the $\mathscr{L}_\infty$-convolution algebra $\hom^{\alpha}(C,\g)$, where $C$ is a $\C_{\infty}^{\vee}$-coalgebra model for $K$, and $\g$ an $\mathscr{L}_{\infty}$-algebra model for $L$. This improves \cite[Thm. 1.4]{ber15}  by no longer requiring that $C$  is strictly cocommutative, but also allowing $C$  to be just cocommutative up to homotopy.

\medskip

Along the way, we prove the folklore result that --- under some assumptions --- the dual of a  commutative rational model of a space is a cocommutative rational model, and dually, that the dual of a cocommutative rational model of a space is a commutative rational model. This result is \cref{thm:dualization of models}. We were not able to find a proof of this result in the literature.

\subsection*{Structure of the paper}

We start this paper with a short recollection on some notions in operad theory in \cref{sect:recollections}. In \cref{sect:inftyMorphisms}, we recall the standard notion of $\infty$-morphisms, before introducing the more general one we need for the subsequent results and studying its homotopical properties in detail. In \cref{sect:convolutionHomotopyAlg}, we recall the construction of convolution homotopy Lie algebras. We get to the core of the present article with the proof of the fact that $\infty$-morphisms of (co)algebras induce $\infty$-morphisms between convolution homotopy Lie algebras in \cref{sect:functoriality} and the proof that this cannot be done simultaneously for algebras and coalgebras in \cref{sect:counterexample}. In \cref{sect:differentSettings}, we explain what happens to the results if one considers non-symmetric operads instead of symmetric operads, or tensor products instead of hom spaces. In \cref{sect:MC spaces}, we prove that the Maurer--Cartan spaces of the convolution $\mathscr{L}_\infty$-algebras are invariant up to homotopy under certain morphisms. Finally, in \cref{sect:application} we apply the theory we developed to rational homotopy theory by constructing rational models for mapping spaces.

\subsection*{Acknowledgments}

Both authors are extremely thankful to Alexander Berglund and Bruno Vallette for their support and their useful comments.  The first author is grateful to the University of Stockholm for the hospitality during a short visit, during which the present collaboration was started. The authors would also like to thank Philip Hackney for answering some of their questions. The first author was a PhD student at Universit\'e Paris 13 at the time of writing, and he was supported by grants from R\'egion Ile-de-France and the grant ANR-14-CE25-0008-01 project SAT. Most of this paper was written while the second author was a PhD student at Stockholm University and it was completed while he was a postdoctoral student at Charles University in Prague; the second author also acknowledges the support of the grant GA CR No. P201/12/G028.

\subsection*{Notations and conventions}

Throughout the article, we work over a field $\k$ of characteristic $0$, unless explicitly stated differently. All operads, cooperads, algebras, and coalgebras are over chain complexes unless explicitly stated otherwise. All operads and cooperads are symmetric unless otherwise specified.

\medskip

In this paper, we mainly work with \emph{shifted} homotopy Lie algebras and \emph{shifted} homotopy associative algebras, which we denote $\L_\infty$-algebras and $\A_\infty$-algebras respectively, instead of their non-shifted versions. These notions have already been used in the literature, see e.g. \cite{dr17} and \cite{dr15}, and have the advantage of greatly simplifying the signs appearing in the computations. One can immediately pass to non-shifted homotopy Lie and associative algebras via a suspension.

\medskip

The symbol $\circ$ is reserved for the monadic composition of $\S$-modules and (co)ope\-rads, as well for the composite product of morphisms of (co)operads, see \cite[Sect. 5.1.4, 5.1.9]{lodayvallette}. The usual composite of e.g. two function $f$ and $g$ will simply be denoted by $gf$.

\medskip

Let $G$ be a finite group and let $V$ be a vector space with left $G$-action. Then we identify invariants and coinvariants of $V$ via the isomorphisms
\[
V^G\longrightarrow V_G,\qquad v\longmapsto[v]
\]
looking at the space of invariants as a subspace of $V$ and taking the class of an element to get to coinvariants. Its inverse is
\[
V_G\longrightarrow V^G,\qquad [v]\longmapsto\frac{1}{|G|}\sum_{g\in G}g\cdot v\ ,
\]
which takes a representative of a coinvariant and averages over its $G$-orbit. This will be used implicitly throughout the text: whenever group actions come into play (e.g. with symmetric operads) we place ourselves in characteristic $0$ and always work with coinvariants.

\medskip

If $V,W$ are two chain complexes, we denote by $\hom(V,W)$ the chain complex of linear maps from $V$ to $W$. A linear map $f$ has degree $n$ if it sends elements of degree $k$ in $V$ to elements of degree $k+n$ in $W$. The differential of $\hom(V,W)$ is given by
\[
\partial(f)\coloneqq d_Wf - (-1)^{|f|}fd_V\ .
\]
We also denote by $V^\vee\coloneqq\hom(V,\k)$ the dual chain complex of $V$, where $\k$ is seen as a chain complex concentrated in degree $0$.

\medskip

A chain complex is \emph{of finite type} if it is finite dimensional in each degree and if the set of degrees where it is non-zero is bounded either above or below. An algebra (of any type) is \emph{of finite type} if the underlying chain complex is. This notion is useful when dualizing algebras. For example, the linear dual of a commutative algebra of finite type is a cocommutative coalgebra, while this is not true if we drop the finite type assumption. If moreover the algebra is concentrated in strictly positive or strictly negative degree, then its dual is conilpotent.

\medskip

Throughout the whole article, all coalgebras and cooperads are implicitly supposed to be conilpotent, see \cite[Sect. 5.8.4]{lodayvallette}, and all operads and cooperads are assumed to be \emph{reduced}, i.e. $0$ in arity $0$ and $\k\id$ in arity $1$.

\section{Recollections} \label{sect:recollections}

In this section, we recall some basic notions of operad theory which we will need throughout the text. We will try to stay as close as possible to the conventions of the book \cite{lodayvallette}, to which we refer the reader for further details. We assume a basic knowledge of the definitions of operads, cooperads, and (co)algebras over (co)operads.

\subsection{The Koszul sign rule and the Koszul sign convention} \label{subsect:KoszulSignRule}

The Koszul sign rule is a sign that is put on the switching maps in the (symmetric monoidal) category of graded vector spaces. Namely, if $V,W$ are two graded vector spaces, then the isomorphism
\[
V\otimes W\longrightarrow W\otimes V
\]
is given by
\[
v\otimes w\longmapsto(-1)^{|v||w|}w\otimes v
\]
on homogeneous elements. The Koszul sign convention is the following convention on maps between graded vector spaces. Let $V_1,V_2,W_1,W_2$ be graded vector spaces, and let $f_i:V_i\to W_i$ be linear maps of homogeneous degree. Then the map $f_1\otimes f_2$ is given by
\[
(f_1\otimes f_2)(v_1\otimes v_2) = (-1)^{|f_2||v_1|}f_1(v_1)\otimes f_2(v_2)\ .
\]
This gives an automatic way of obtaining the correct signs in computations without having to actually apply maps to elements. An example of an application of the Koszul sign rule is the following. Let $(V,d_V)$ be a chain complex, then the differential on the suspended complex $(sV)_n\coloneqq V_{n-1}$ is given by
\[
d_{sV}(sv) = -sd_Vv\ .
\]
Let $V$ be a graded vector space, and let $v_1,\ldots,v_n$ be elements of $V$. Suppose we are given a partition of $[n]\coloneqq\{1,2,\ldots,n\}$ into disjoint (ordered) subsets $S_1\sqcup\ldots\sqcup S_k$. We will denote by $\sh(v^{S_1},\ldots,v^{S_k})$ the sign obtained by Koszul sign rule when doing the switch, that is to say that we have
\[
v_1\otimes\cdots\otimes v_n = (-1)^{\sh(v^{S_1},\ldots,v^{S_k})}\bigotimes_{i=1}^k\left(\bigotimes_{s\in S_i}v_s\right)
\]
in $(V^{\otimes n})_{\S_n}$. For example, suppose $n=3$ and $S_1=\{2\},S_2=\{1,3\}$. Then
\[
v_1\otimes v_2\otimes v_3 = (-1)^{|v_1||v_2|}v_2\otimes v_1\otimes v_3\ ,
\]
so that $\sh(v^{S_1},v^{S_2}) \equiv |v_1||v_2|\mod2$.

\subsection{Notations for operads}

As already stated, we reserve the symbol $\circ$ for operadic composition. For example, if $M_1,M_2,N_1,N_2$ are $\S$-modules, and $f_i:M_i\to N_i$, then
\[
f_1\circ f_2:M_1\circ M_2\longrightarrow N_1\circ N_2
\]
is given by
\[
(f_1\circ f_2)(\mu;\nu_1,\ldots,\nu_k) = (f_1(\mu);f_2(\nu_1),\ldots,f_2(\nu_k))\ .
\]
See \cite[Sect. 5.1]{lodayvallette} for details. We will also use the notation $M\circ(N_1;N_2)$ for the sub-$\S$-module of $M\circ(N_1\oplus N_2)$ which is linear in $N_2$, as well as
\[
M\circ_{(1)}N\coloneqq M\circ(I;N)\ ,
\]
and
\[
f_1\circ_{(1)}f_2:M_1\circ_{(1)}M_2\longrightarrow N_1\circ_{(1)}N_2
\]
for the map acting by $f_1$ on $M_1$ and by $f_2$ on $M_2$. Furthermore, we denote by $\mu\otimes_j\nu$ the element
\[
\mu\otimes_j\nu\coloneqq(\mu;1,\ldots,1,\underbrace{\nu}_{j^{\text{th}}\text{ position}},1,\ldots,1)\in M\circ_{(1)}N\ .
\]
Finally, if $f:M_1\to M_2$ and $g:N_1\to N_2$, then their \emph{infinitesimal composite} is the map
\[
f\circ' g:M_1\circ N_1\longrightarrow M_2\circ(N_1;N_2)
\]
given by applying $f$ to $M_1$ and $g$ to exactly one $N_2$ component, in all possible ways. See \cite[Sect. 6.1]{lodayvallette} for a more detailed exposition.

\subsection{Weight grading}\label{subsect:weight grading}

In order to do homological algebra on (co)operads, one often needs an additional grading different from the homological degree, see \cite[Sect. 6.3.11]{lodayvallette}. This allows the use of spectral sequence arguments in the proof of certain results, see e.g. \cite[Sect. 6.7]{lodayvallette}.

\begin{definition}
	A \emph{connected weight grading} on an operad $\P$ is a decomposition
	\[
	\P = \k\id\oplus\P^{(1)}\oplus\P^{(2)}\oplus\P^{(3)}\oplus\cdots
	\]
	into subspaces $\P^{(k)}$ of \emph{weight $k$}, such that both the differential and the composition map preserve the total weight. Similarly, a \emph{connected weight grading} on a cooperad $\C$ is an analogue decomposition such that the total weight is preserved by the decomposition map.
\end{definition}

Morphisms between connected weight graded (co)operads are required to preserve the weight, as are twisting morphisms.

\medskip

Since we work under the standing assumption that our (co)operads are reduced, we have a canonical connected weight grading given by
\[
\P^{(k)}\coloneqq\P(k+1)
\]
for any operad $\P$, and similarly for cooperads. Most of the results of the present paper are expected to hold in the general setting of connected weight graded (co)operads.

\subsection{Convolution operads}

Let $\C$ be a cooperad, and let $\P$ be an operad. We define an $\S$-module $\hom(\C,\P)$ by
\[
\hom(\C,\P)(n)\coloneqq\hom(\C(n),\P(n))
\]
for $n\ge0$ and with the action of the symmetric groups given by conjugation:
\[
(f^\sigma)(x) = f(x^{\sigma^{-1}})^\sigma\ .
\]
In \cite[p. 3]{bm03}, this $\S$-module was endowed with an operad structure as follows (see also \cite[Sect. 6.4.1]{lodayvallette} for an exposition using the same notations as in the present paper). Let $f\in\hom(\C,\P)(k)$ and $g_i\in\hom(\C,\P)(n_i)$ for $i=1,\ldots,k$. Let $n=n_1+\cdots+n_k$. Then $\gamma(f\otimes(g_1,\ldots,g_k))$ is given by the composite
\begin{align*}
	\C(n)\xrightarrow{\Delta}(\C\circ\C)(n)\xrightarrow{\proj}\C(k)\otimes\C(n_1)\otimes\cdots\otimes\C(n_k)\otimes\k[\S_n]\\
	\xrightarrow{f\otimes g_1\otimes\cdots\otimes g_k\otimes1}\P(k)\otimes\P(n_1)\otimes\cdots\otimes\P(n_k)\otimes\k[\S_n]\xrightarrow{\gamma}\P(n)
\end{align*}

\begin{definition}
	The operad $\hom(\C,\P)$ defined above is called the \emph{convolution operad} of $\C$ and $\P$.
\end{definition}

We have the following obvious fact, already proven e.g. in \cite[Prop. 2.9]{rn17} or \cite[Prop. 7.1]{w16}.

\begin{lemma}
	Let $D$ be a $\C$-coalgebra, and let $A$ be a $\P$-algebra. Then the chain complex $\hom(D,A)$ is naturally a $\hom(\C,\P)$-algebra.
\end{lemma}

\subsection{Operadic twisting morphisms}

Let $\C$ be a cooperad, and let $\P$ be an operad. Then one can endow the chain complex of morphisms of $\S$-modules (i.e. $\S$-equivariant maps)
\[
\hom_\S(\C,\P)\coloneqq\prod_{n\ge0}\hom_{\S_n}(\C(n),\P(n))
\]
with a Lie algebra structure. First, one defines a pre-Lie product $\star$ given as the composite
\[
f\star g \coloneqq \left(\C\xrightarrow{\Delta_{(1)}}\C\circ_{(1)}\C\xrightarrow{f\circ_{(1)}g}\P\circ_{(1)}\P\xrightarrow{\gamma_{(1)}}\P\right)\ .
\]
One then antisymmetrizes this product to obtain a Lie bracket
\[
[f,g]\coloneqq f\star g - (-1)^{|f||g|}g\star f\ .
\]

\begin{remark}
	A more general construction associates a pre-Lie algebra, and thus a Lie algebra, to any operad, see \cite[Sect. 5.4.3]{lodayvallette}.
\end{remark}

\begin{definition}
	A degree $-1$ element $\alpha\in\hom_\S(\C,\P)$ solving the Maurer--Cartan equation
	\[
	\partial(\alpha)+\frac{1}{2}[\alpha,\alpha] = 0
	\]
	in the Lie algebra defined above is called an \emph{(operadic) twisting morphism} from $\C$ to $\P$. The set of all such twisting morphisms is denoted by $\Tw(\C,\P)$.
\end{definition}

Twisting morphisms play a very important role in operad theory, in particular with respect to the bar and cobar construction for (co)operads.

\begin{theorem}[\cite{lodayvallette} Theorem 6.5.7 ]
	Let $\C$ be a cooperad, let $\P$ be an operad. There exist natural isomorphisms
	\[
	\hom_{\mathsf{Op}}(\Omega\C,\P)\cong\Tw(\C,\P)\cong\hom_{\mathsf{coOp}}(\C,\bar\P)\ .
	\]
\end{theorem}

We reserve the symbols $\kappa$, $\pi$, and $\iota$ for the following canonical twisting morphisms:
\begin{itemize}
	\item $\kappa:\P^{\antishriek}\to\P$ for the twisting morphism given by Koszul duality, see e.g. \cite[Sect. 7.4]{lodayvallette},
	\item $\pi:\bar\P\to\P$ for the universal twisting morphism associated to the counit of the bar-cobar adjunction, and
	\item $\iota:\C\to\Omega\C$ for the universal twisting morphism associated to the unit of the bar-cobar adjunction.
\end{itemize}
For more details see \cite[Sect. 6.5.4]{lodayvallette}. The specific (co)operads $\C,\P$ will vary from case to case, and will be omitted from the notation.

\subsection{Bar and cobar construction relative to a twisting morphism}

Let $\C$ be a cooperad, and let $\P$ be an operad. Suppose we are given a twisting morphism $\alpha:\C\to\P$. Then one has two adjoint functors, called the \emph{bar} and \emph{cobar construction} respectively
\[
\adjunction{\Omega_\alpha}{\mathsf{conil.\ }\C\text{-}\mathsf{cog.}}{\P\text{-}\mathsf{alg.}}{\bar_\alpha}
\]
between conilpotent $\C$-coalgebras and $\P$-algebras. They are defined as follows.
\begin{enumerate}
	\item The bar construction $\bar_\alpha A$ of a $\P$-algebra $A$ is the cofree $\C$-coalgebra $\C(A)$ with the square zero codifferential $d_{\bar_\alpha A}\coloneqq d_1+d_2$, where $d_1\coloneqq d_\C\circ1_A + 1_\C\circ'd_A$ and $d_2$ is the unique coderivation extending the composite
	\[
	\C(A)\xrightarrow{\alpha\circ1_A}\P(A)\xrightarrow{\gamma_A}A\ .
	\]
	That is to say, the full expression for $d_2$ is given by the composite
	\begin{align*}
	\C(A)&\xrightarrow{\Delta_{(1)}\circ1_A}(\C\circ_{(1)}\C)(A)\\
	&\xrightarrow{(1_\C\circ_{(1)}\alpha)\circ1_A}(\C\circ_{(1)}\P)(A)\cong\C\circ(A;\P(A))\xrightarrow{1_\C\circ(1_A;\gamma_A)}\C(A)\ .
	\end{align*}
	\item Dually, the cobar construction $\Omega_\alpha D$ of a conilpotent $\C$-coalgebra $D$ is the free $\P$-algebra $\P(D)$ with the differential $d_{\Omega_\alpha D}\coloneqq d_1+d_2$, where $d_1\coloneqq d_\P\circ1_D + 1_\P\circ'd_D$ and $-d_2$ is the unique derivation extending the composite
	\[
	D\xrightarrow{\Delta_D}\C(D)\xrightarrow{\alpha\circ1_D}\P(D)\ .
	\]
	Similarly to the previous case, the full expression for $-d_2$ is given by the composite
	\begin{align*}
	\P(D)&\xrightarrow{1_\P\circ'\Delta_D}\P\circ(D;\C(D))\\
	&\xrightarrow{1_\P\circ(1_D;\alpha\circ1_D)}\P\circ(D;\P(D))\cong(\P\circ_{(1)}\P)(D)\xrightarrow{\gamma_{(1)}\circ1_D}\P(D)\ .
	\end{align*}
\end{enumerate}

If we fix a $\C$-coalgebra $D$ and a $\P$-algebra $A$, to an operadic twisting morphism one can also associate a Maurer--Cartan equation on the chain complex $\hom(D,A)$ of linear maps from $D$ to $A$. It is given by
\begin{equation} \label{eq:MCrelToAlpha}
	\partial(\varphi)+\star_\alpha(\varphi) = 0
\end{equation}
for $\varphi\in\hom(D,A)$, where $\partial$ is the differential of $\hom(D,A)$, given as usual by
\[
\partial(\varphi) = d_A\varphi - (-1)^{|\varphi|}\varphi d_D\ ,
\]
and $\star_\alpha$ is the operator defined by
\[
\star_\alpha(\varphi) = \big(D\xrightarrow{\Delta_D}\C\circ D\xrightarrow{\alpha\circ\varphi}\P\circ A\xrightarrow{\gamma_A}A\big)\ .
\]
The degree $0$ elements solving \cref{eq:MCrelToAlpha} are called \emph{twisting morphisms (relative to $\alpha)$}. The set of all such twisting morphisms is denoted by $\Tw_\alpha(D,A)\subseteq\hom(D,A)_0$. The following result is found e.g. in \cite[Prop. 11.3.1]{lodayvallette}.

\begin{theorem} \label{thm:RosettaStone}
	For any $\C$-coalgebra $D$ and any $\P$-algebra $A$, there are natural bijections
	\[
	\hom_{\P\text{-}\mathsf{alg}}(\Omega_\alpha D,A)\cong\Tw_\alpha(D,A)\cong\hom_{\C\text{-}\mathsf{cog}}(D,\bar_\alpha A)\ .
	\]
	In particular, the functors $\Bar_\alpha$ and $\Omega_\alpha$ are adjoint.
\end{theorem}

For more details, the reader is invited to consult the standard reference \cite[Sect. 11.2]{lodayvallette}.

\subsection{(Shifted) homotopy Lie algebras}

A type of algebras which plays a central role in the present article are \emph{$\L_\infty$-algebras}, i.e. shifted (strong) homotopy Lie algebras. There are many books and articles giving good introductions to $\L_\infty$-algebras, and we recommend e.g. \cite[Sect. 2]{dr17} for details in the notions, since they use the same conventions as in the present paper (except they work with cochain complexes instead of chain complexes), and \cite[Sect. 13.2.9--13]{lodayvallette} (in the non-shifted setting). We give nonetheless a brief review of the structure of this kind of algebras.

\begin{remark}
	The theory of homotopy Lie algebras and suspended homotopy Lie algebras are exactly the same: one is sent to the other by (de)suspension.
\end{remark}

\begin{definition} \label{def:LooAlg}
	An \emph{$\L_\infty$-algebra} is a chain complex $\g$ endowed with graded symmetric operations
	\[
	\ell_n:\g^{\otimes n}\longrightarrow\g\ ,\qquad n\ge2\ ,
	\]
	of degree $-1$ satisfying the relations
	\[
	\sum_{\substack{n_1+n_2 = n+1\\\sigma\in\Sh(n_2,n_1-1)}}(\ell_{n_1}\circ_1\ell_{n_2})^\sigma = 0
	\]
	for $n\ge1$, where we used the notation $\ell_1 = d_\g$. We usually speak of ``the $\L_\infty$-algebra $\g$", without specifying the operations $\ell_n$.
\end{definition}

The relation for $n=2$ gives us (a shifted version of) the Leibniz rule for $\ell_2$, the relation for $n=3$ tells us that the (shifted) Jacobi rule is satisfied up to a homotopy given by $\ell_3$, while the relations for $n\ge4$ are coherent higher homotopies for the operations.

\medskip

One can of course consider strict morphisms of $\L_\infty$-algebras, i.e. chain maps
\[
\phi:\g\longrightarrow\h
\]
such that
\[
\phi\circ\ell_n = \ell_n\circ(\phi,\ldots,\phi)
\]
for all $n\ge1$. However, this definition of morphism is too strong for some applications. Therefore, one relaxes the notion of morphisms in such a way that they commute with the $\L_\infty$-structures only up to a system of coherent homotopies.

\begin{definition}
	Let $\g$ and $\h$ be $\L_\infty$-algebras. An \emph{$\infty$-morphism} $\Psi$ from $\g$ to $\h$, which we denote by $\Psi:\g\rightsquigarrow\h$, is a sequence of linear maps
	\[
	\psi_n:\g^{\otimes n}\longrightarrow\h
	\]
	for $n\ge1$ such that
	\[
	\sum_{\substack{n_1+n_2 = n+1\\1\le j\le n_1\\\sigma\in\Sh(n_2,n_1-1)}}(\psi_{n_1}\circ_i\ell_{n_2})^\sigma = \sum_{\substack{k\ge1\\i_1+\cdots+i_k = n\\\tau\in\Sh(i_1,\ldots,i_k)}}(\ell_k\circ(\psi_{i_1},\ldots,\psi_{i_k}))^\tau\ .
	\]
\end{definition}

Of course, $\L_\infty$-algebras are homotopy $\susp\otimes\lie$-algebras in the sense of \cite[Sect. 7]{lodayvallette}, and $\infty$-morphisms agree with the notion defined in loc. cit., as well as in \cref{subsec:inftymorhpism}.

\section{Infinity-morphisms relative to twisting morphisms} \label{sect:inftyMorphisms}

The notion of $\infty$-morphisms between homotopy algebras over a Koszul operad is well established in the literature, see e.g. \cite[Sect. 10.2]{lodayvallette}. In this section, we recall briefly the classical theory of $\infty$-morphisms before introducing a generalized version of $\infty$-morphisms of algebras and coalgebras relative to a twisting morphism of operads and studying it in depth. It is worth noticing that this more general notion of $\infty$-morphisms has already appeared in the literature, where it was first introduced in \cite{Ber14}.

\subsection{Infinity-morphisms of homotopy algebras} \label{subsec:inftymorhpism}

We recall the definition and some important properties of $\infty$-morphisms between homotopy $\P$-algebras, where $\P$ is a Koszul operad. The notion of $\infty$-morphism is more relevant than strict morphisms in the homotopy theory of algebras.

\medskip

For the rest of this subsection, let $\P$ be a Koszul operad. Recall that a $\P_\infty$-algebra, or homotopy $\P$-algebra, is an algebra over the operad $\P_\infty\coloneqq\Omega\P^{\antishriek}$, see \cite[Sect. 7, 10.1]{lodayvallette}.

\begin{definition} \label{def:classicalInftyMorphisms}
	Let $A$ and $A'$ be two $\P_\infty$-algebras. An $\infty$-morphism $\Psi: A \rightsquigarrow A'$ is a morphism
	\[
	\Psi:B_\iota A \longrightarrow B_\iota A'
	\]
	between the bar constructions relative to the canonical twisting morphism $\iota:\P^{\antishriek}\to\P_\infty$. Composition of $\infty$-morphisms is given by the usual composition of morphisms of $\P^{\antishriek}$-coalgebras between the bar constructions. The category of $\P$-algebras with $\infty$-morphisms is denoted by $\infty\text{-}\P_\infty\text{-}\mathsf{alg}$.
\end{definition}

\begin{remark}
	Notice that $\P$-algebras are special cases of $\P_\infty$-algebras. If $A,A'$ are $\P$-algebras, then an $\infty$-morphism $\Psi:A\rightsquigarrow A'$ is the same thing as a morphism
	\[
	\Psi:B_\kappa A \longrightarrow B_\kappa A'\ .
	\]
	This can be seen e.g. by \cref{lemma:equalityOfInftyMorphisms}.
\end{remark}

Now let $\Psi:A\rightsquigarrow A'$ be an $\infty$-morphism of $\P_\infty$-algebras. Then, thanks to \cref{thm:RosettaStone} we know that $\Psi$ is equivalent to an element of $\Tw_\iota(\bar_\iota A,A')$, which we will denote again by $\Psi$ by abuse of notation. This is equivalent to a collection of linear maps
\[
\psi_n:\P^{\antishriek}(n)\otimes_{\S_n}A^{\otimes n}\longrightarrow A'
\]
which satisfy certain compatibilities (equivalent to the Maurer--Cartan equation). Since $\P^{\antishriek}(1)\cong\k$, we can see $\psi_1$ as a linear map
\[
\psi_1:A\longrightarrow A'\ .
\]
Strict morphisms of $\P_\infty$-algebras are special cases of $\infty$-morphisms. Indeed, it is straightforward to see the following.

\begin{lemma}
	A morphism of $\P_\infty$-algebras $\psi:A\to A'$ is equivalent to an $\infty$-morphism $\Psi:A\rightsquigarrow A'$ with $\psi_1=\psi$ and $\psi_n=0$ for all $n\ge2$.
\end{lemma}

The compatibility conditions on the whole collection $\{\psi_n\}_n$ associated to an $\infty$-morphism vary depending on the operad $\P$ one considers. However, something can always be said about $\psi_1$.

\begin{lemma}
	The component $\psi_1:A\to A'$ of an $\infty$-morphism is a chain map.
\end{lemma}

\begin{proof}
	This follows straightforwardly from the Maurer--Cartan equation.
\end{proof}

\begin{definition}
	An $\infty$-morphism $\Psi:A\rightsquigarrow A'$ of $\P_\infty$-algebras is called an $\infty$-iso\-morphism, resp. an $\infty$-quasi-isomorphism, if $\psi_1$ is an isomorphism of chain complexes, resp. a quasi-isomorphism.
\end{definition}

It is straightforward to show that the $\infty$-isomorphisms are exactly the isomorphisms in the category $\infty\text{-}\P_\infty\text{-}\mathsf{alg}$. The $\infty$-quasi-isomorphisms have a similar property at the level of homology.

\begin{theorem}[{\cite[Thm. 10.4.4]{lodayvallette}}]
	Let $\Psi:A\rightsquigarrow A'$ be an $\infty$-quasi-isomorphisms of $\P_\infty$-algebras. Then there exists another $\infty$-quasi-isomorphism $A'\rightsquigarrow A$ which is inverse to $\Psi$ in homology.
\end{theorem}

\subsection{Infinity-morphisms relative to a twisting morphism}

In the definition of an $\infty$-morphism in \cref{subsec:inftymorhpism}, an $\infty$-morphism was defined as a map between the bar constructions relative to the canonical twisting morphism
\[
\iota:\P^{\antishriek}\to\Omega\P^{\antishriek}\ .
\]
We will replace $\iota$ by a more general twisting morphism $\alpha$ to get $\infty$-morphisms relative to $\alpha$. Versions of this idea have been defined implicitly in \cite{Mar04} and this definition can also be found in \cite{Ber14}. Notice that this allows us for example to consider $\infty$-morphisms for algebras over non-Koszul operads, as well as for coalgebras.

\begin{definition}\label{def:inftymorphism1}
	Let $\C$ be a cooperad, let $\P$ be an operad, and let $\alpha:\C\to\P$ be a Koszul twisting morphism.
	\begin{enumerate}
		\item An \emph{$\infty$-morphism of $\P$-algebras relative to $\alpha$}, or an \emph{$\infty_\alpha$-morphism of $\P$-algebras}, between two $\P$-algebras $A$ and $A'$ is a morphism $\Psi$ of $\C$-coalgebras
		\[
		\Psi:\Bar_\alpha A\longrightarrow\Bar_\alpha A'\ .
		\]
		Composition of $\infty_\alpha$-morphisms of $\P$-algebras is given by the standard composition of morphisms of $\C$-coalgebras between the bar constructions. We denote the category of $\P$-algebras with $\infty_\alpha$-morphisms by $\infty_\alpha\text{-}\Palg$.
		\item An \emph{$\infty$-morphism of conilpotent $\C$-coalgebras relative to $\alpha$}, or an \emph{$\infty_\alpha$-mor\-phism of conilpotent $\C$-coalgebras}, between two conilpotent $\C$-algebras $D'$ and $D$ is a morphism $\Phi$ of $\P$-algebras
		\[
		\Phi:\Omega_\alpha D'\longrightarrow\Omega_\alpha D\ .
		\]
		Composition of $\infty_\alpha$-morphisms of $\C$-coalgebras is given by the standard composition of morphisms of $\P$-algebras between the cobar constructions. We denote the category of conilpotent $\C$-coalgebras with $\infty_\alpha$-morphisms by $\infty_\alpha\text{-}\Ccog$.
	\end{enumerate}
\end{definition}

Since $B_{\alpha}A'$ is a cofree $\C$-coalgebra every morphism $\Psi:B_{\alpha}A \rightarrow B_{\alpha}A'$ is completely determined by its image on the cogenerators of $B_{\alpha}A'$. Similarly, since $\Omega_{\alpha}D '$ is free  every morphism $\Phi:\Omega_{\alpha}D' \rightarrow \Omega_{\alpha}D$ is determined by the image of the generators of $\Omega_{\alpha} D '$.

\begin{remark}
	Notice that, if $\P$ is a Koszul operad, then the $\infty$-morphisms of $\P_\infty$-algebras of \cref{def:classicalInftyMorphisms} are exactly $\infty_\iota$-morphisms of $\P_\infty$-algebras for the canonical twisting morphism
	\[
	\iota:\P^{\antishriek}\longrightarrow\Omega\P^{\antishriek}=\P_\infty  .
	\]
\end{remark}

Now we prove a straightforward lemma that allows us in some cases to relate $\infty_\alpha$-morphisms for different twisting morphisms $\alpha$. If $f:\C_2\to\C_1$ is a morphism of cooperads and $D$ is a $\C_2$-coalgebra, we denote by $f_*D$ the same chain complex with the $\C_1$-coalgebra structure obtained by pushing forward its $\C_2$-coalgebra structure by $f$. Dually, if $g:\P_1\to\P_2$ is a morphism of operads and $A$ is a $\P_2$-algebra, we denote by $g^*A$ the chain complex $A$ seen as a $\P_1$-algebra by pulling back its original structure.

\begin{lemma} \label{lemma:equalityOfInftyMorphisms}
	Let $\C',\C$ be two cooperads and let $\P$ be an operad. Let $\alpha\in\Tw(\C,\P)$, let $f:\C'\to\C$ be a morphism of cooperads, and let $D$ be a conilpotent $\C'$-coalgebra. Then
	\[
	\Omega_\alpha(f_*D) = \Omega_{f^*\alpha}D\ .
	\]
	In particular, $\infty_{f^*\alpha}$-morphisms between conilpotent $\C'$-coalgebras are the same as $\infty_\alpha$-morphisms between the same coalgebras seen as $\C$-coalgebras by pushforward of the structure along $f$.
	
	\medskip
	
	Dually, let $\C$ be a cooperad and let $\P,\P'$ be two operads. Let $\alpha\in\Tw(\C,\P)$, let $g:\P\to\P'$ be a morphism of operads, and let $A$ be a $\P'$-algebra. Then
	\[
	\bar_\alpha(g^*A) = \bar_{g_*\alpha}A\ .
	\]
	In particular, $\infty_{g_*\alpha}$-morphisms between $\P'$-algebras are the same as $\infty_\alpha$-morphisms between the same algebras seen as $\P$-algebras by pullback of the structure along $g$.
\end{lemma}

\begin{proof}
	We only prove the first of the two facts, the proof of the second one being dual. As algebras over graded vector spaces, it is clear that we have
	\[
	\Omega_{f^*\alpha}D = \P(D) = \Omega_{\alpha}(f_*D)\ ,
	\]
	so that we only have to check that the differentials agree.	We have
	\[
	d_{\Omega_{f^*\alpha}D} = d_{\P(D)} + d^{f^*\alpha}_2
	\]
	with $d^{f^*\alpha}_2$ given by the composite
	\begin{center}
		\begin{tikzpicture}
		\node (a) at (0,0){$\P(D)$};
		\node (b) at (4,0){$\P\circ(D;\C'\circ D)$};
		\node (c) at (4,-2){$\P\circ(D;\P\circ D)\cong(\P\circ_{(1)}\P)(D)$};
		\node (d) at (9.2,-2){$\P(D)\ .$};
		
		\draw[->] (a) to node[above]{\small$1_\P\circ'\Delta_D$} (b);
		\draw[->] (b) to node[right]{\small$1_\P\circ(1_D;f^*\alpha\circ1_D)$} (c);
		\draw[->] (c) to node[above]{\small$\gamma_{(1)}\circ1_D$} (d);
		\end{tikzpicture}
	\end{center}
	The part $d_{\P(D)}$ is independent of the twisting morphism, and thus of no interest to us. For the other part, we notice that
	\begin{align*}
		(1_\P\circ(1_D;f^*\alpha\circ1_D))(1_\P\circ'\Delta_D) =&\ (1_\P\circ(1_D;\alpha f\circ1_D))(1_\P\circ'\Delta_D)\\
		=&\ (1_\P\circ(1_D;\alpha\circ1_D))(1_\P\circ'(f\circ 1_D)\Delta_D)\\
		=&\ (1_\P\circ(1_D;\alpha_1\circ1_D))(1_\P\circ'\Delta_{f_*D})\ ,
	\end{align*}
	which implies the result.
\end{proof}

\subsection{Homotopical properties of infinity-morphisms}\label{subsect:oo-alpha-qis and rectifications}

The homotopy theory of classical $\infty$-morphisms is already well known in the literature, see for example \cite[Sect. 11.4]{lodayvallette} and \cite{val14}. We give here some results in the analogous theory for $\infty_\alpha$-morphisms, which we will need in \cref{sect:application}.

\medskip

Fix a Koszul morphism $\alpha:\C\to\P$ from a cooperad $\C$ to an operad $\P$. Notice that, since we assumed that all the (co)operads are reduced, for any $\infty_\alpha$-morphism $\Psi:A\rightsquigarrow A'$ of $\P$-algebras we have a canonical chain map
\[
\psi_1:A\longrightarrow A'\ .
\]
With this property, we can now define notions of $\infty_\alpha$-quasi-isomorphisms for (co)algebras. The version for algebras generalizes the one defined for classical $\infty$-morphisms, while the version for coalgebras is slightly different, but coherent with the notion of weak equivalences introduced in the article \cite{val14}.

\begin{definition}
	Let $\C$ be a cooperad, let $\P$ be an operad, and let $\alpha:\C\to\P$ be a twisting morphism.
	\begin{enumerate}
		\item An $\infty_\alpha$-morphism of $\P$-algebras $\Psi:A\rightsquigarrow A'$ is an \emph{$\alpha$-weak equivalence} if the morphism
		\[
		\Psi:\bar_\alpha A\longrightarrow \bar_\alpha A'
		\]
		is a weak equivalence of coalgebras in the category of conilpotent $\C$-coalgebras with the Vallette model structure \cite{val14}, i.e. if
		\[
		\Omega_\alpha\Psi:\Omega_\alpha\bar_\alpha A\longrightarrow\Omega_\alpha\bar_\alpha A'
		\]
		is a quasi-isomorphism.
		\item An $\infty_\alpha$-morphism of $\P$-algebras $\Psi:A\rightsquigarrow A'$ is an \emph{$\infty_\alpha$-quasi-isomorphism} if the chain map
		\[
		\psi_1:A\longrightarrow A'
		\]
		is a quasi-isomorphism.
		\item An $\infty_\alpha$-morphism of $\C$-coalgebras $\Phi:D'\rightsquigarrow D$ is an \emph{$\alpha$-weak equivalence} if the morphism
		\[
		\Phi:\Omega_\alpha D'\longrightarrow\Omega_\alpha D
		\]
		is a quasi-isomorphism, i.e. if it is a weak equivalence in the classical Hinich model structure on the category of $\P$-algebras \cite{Hin97homological}.
		\item An $\infty_\alpha$-morphism of $\C$-coalgebras $\Phi:D'\rightsquigarrow D$ is an \emph{$\infty_\alpha$-quasi-isomorphism} if the chain map
		\[
		\phi_1:D'\longrightarrow D
		\]
		is a quasi-isomorphism.
	\end{enumerate}
\end{definition}

\begin{remark}
	Notice that in the Vallette model structure on $\C$-coalgebras, the weak equivalences are created by the cobar construction $\Omega_\alpha$, that is $f$ is a weak equivalence of $\C$-coalgebras if, and only if $\Omega_\alpha f$ is a quasi-isomorphism of algebras, i.e. a weak equivalence in the classical Hinich model structure on $\P$-algebras. This motivates the definition of $\alpha$-weak equivalences given above.
\end{remark}

We will now try to understand how these four notions are related to each other. We begin with a classical fact. It was originally stated for $\P$ a Koszul operad and the classical $\infty$-morphisms of homotopy $\P$-algebras, but the proof readily generalizes to our setting. See also \cite[Prop. 32]{legrignou16}.

\begin{theorem}[{\cite[Prop. 11.4.7]{lodayvallette}}]
	An $\infty_\alpha$-morphism of $\P$-algebras is an $\alpha$-weak equivalence if and only if it is an $\infty_\alpha$-quasi-isomorphism.
\end{theorem}

For coalgebras, we can proceed similarly to prove a slightly weaker statement. In order to do so, we introduce a notion of \emph{rectification} between different types of (co)algebras which generalizes the one found in \cite[Sect. 11.4.3]{lodayvallette}. It will be extremely useful in \cref{sect:application}, where it will allow us to pass between $\infty$-morphisms relative to different twisting morphisms.

\medskip

Let $\C$ be a cooperad, let $\P,\P'$ be operads, let $\alpha:\C\to\P$ be a Koszul twisting morphism, take $g:\P\to\P'$ a quasi-isomorphism, and define $\alpha':\C\to\P'$ by $\alpha'\coloneqq g\alpha$. Notice that $g$ is a quasi-isomorphism if, and only if $\alpha'$ is a Koszul twisting morphism. We define the \emph{rectification functor}
\[
R^{g,\alpha}:\infty_\alpha\text{-}\Palg\longrightarrow\infty_\alpha\text{-}\Palg
\]
by
\[
R^{g,\alpha}(A)\coloneqq g^*\Omega_{\alpha'}\bar_\alpha A
\]
for a $\P$-algebra $A$, and
\[
R^{g,\alpha}(\Psi)\coloneqq g^*\Omega_{\alpha'}\Psi
\]
on $\infty_\alpha$-morphisms. There is a natural transformation $N$ from $R^{g,\alpha}$ to the identity of $\infty_\alpha\text{-}\Palg$ given by
\[
N_A:\bar_\alpha A\xrightarrow{\eta_{\bar_\alpha A}}\bar_{\alpha'}\Omega_{\alpha'}\bar_\alpha A = \bar_\alpha R^{g,a}(A)\ ,
\]
where $\eta$ is the unit of the bar-cobar adjunction relative to $\alpha'$. Notice that $\eta_{\bar_\alpha A}$ is a weak equivalence by \cite[Thm. 2.6]{val14}, so that $N$ defines a natural $\infty_\alpha$-quasi-isomorphism
\[
N_A:A\rightsquigarrow R^{g,\alpha}(A)
\]
by \cite[Prop. 11.4.7]{lodayvallette}.

\medskip

Dually, let $\C,\C'$ be cooperads, let $\P$ be an operad, let $\alpha:\C\to\P$ be a Koszul morphism, take $f:\C'\to\C$ be a quasi-isomorphism, and define $\alpha':\C'\to\P$ by $\alpha'\coloneqq\alpha f$. We define the \emph{rectification functor}
\[
R_{\alpha,f}:\infty_\alpha\text{-}\Ccog\longrightarrow\infty_\alpha\text{-}\Ccog
\]
by
\[
R_{\alpha,f}(D)\coloneqq f_*\bar_{\alpha'}\Omega_\alpha D
\]
on conilpotent $\C$-coalgebras, and
\[
R_{\alpha,f}(\Phi)\coloneqq f_*\bar_{\alpha'}\Phi
\]
on $\infty_\alpha$-morphisms of $\C$-coalgebras. The $f_*$ in the action on morphisms is there only for consistency of notation. There is a natural transformation $E$ from the rectification $R_{\alpha,f}$ to the identity of $\infty_\alpha\text{-}\Ccog$ given by
\[
E_D:\Omega_\alpha R_{\alpha,f}(D) = \Omega_{\alpha'}\bar_{\alpha'}\Omega_\alpha D\xrightarrow{\epsilon_{\Omega_\alpha D}}\Omega_\alpha D\ .
\]
As before, the morphism $\epsilon_{\Omega_\alpha D}$ is a quasi-isomorphism, so that $E$ defines a natural $\alpha$-weak equivalence.

\begin{proposition}\label{prop:rectification is oo-qi to original}
	Let $D$ be a $\C$-coalgebra. The $\alpha$-weak equivalence
	\[
	E_D:R_{\alpha,f}(D)\rightsquigarrow D
	\]
	is an $\infty_\alpha$-quasi-isomorphism.
\end{proposition}

\begin{proof}
	We have to prove that the first component
	\[
	e_1:R_{\alpha,f}(D)\longrightarrow D
	\]
	of $E_D$ is a quasi-isomorphism. We will do this by a spectral sequence argument analogous to the one of \cite[Thm. 11.3.3 and 11.4.4]{lodayvallette}. We start by noticing that
	\[
	e_1:(\C'\circ\P)(D)\longrightarrow D
	\]
	is given by the projection onto $D$. We filter the left-hand side by the number of times that $D$ appears, i.e. by
	\[
	F_p\coloneqq\bigoplus_{k\le p}(\C'\circ\P)(k)\otimes_{\S_k}D^{\otimes k}\ .
	\]
	This filtration is increasing, bounded below and exhaustive. The page $E^0$ of the associated spectral sequence equals $(\C'\circ_{\alpha'}\P)(D)$, since the only parts of the differential that preserve the weight (that is, the arity) are the internal differential of $D$ and the part coming from the twisting morphism $\alpha'$. The page $E^1$ of the spectral sequence is
	\[
	H_\bullet((\C'\circ_{\alpha'}\P)(D))\cong H_\bullet(\C'\circ_{\alpha'}\P)\circ H_\bullet(D)\cong H_\bullet(D)
	\]
	by the operadic Künneth formula \cite[Prop. 6.2.3]{lodayvallette} and the fact that $\alpha'$ is a Koszul morphism. On the other side, we filter $D$ by $F_pD=D$ for $p\ge0$ and $F_pD=0$ otherwise. This filtration is also increasing, bounded below and exhaustive. The map $e_1$ is a map of spectral sequences, and so the induced map at the page $E^1$ is $H_\bullet(e_1)$, which induces an isomorphism. Therefore, the chain map $e_1$ is a quasi-isomorphism.
\end{proof}

Notice that if $f=1_\C$, then the rectification becomes the functor $\bar_\alpha\Omega_\alpha$, and the natural $\infty_\alpha$-morphism $E_D$ is given by the counit $\varepsilon$ of the bar-cobar adjunction (seen as an $\infty_\alpha$-morphism). As a consequence of this result, we have the following.

\begin{theorem}\label{thm: alpha-we are oo-alpha-qi}
	If an $\infty_\alpha$-morphism of $\C$-coalgebras is an $\alpha$-weak equivalence, then it is an $\infty_\alpha$-quasi-isomor\-phism.
\end{theorem}

Notice that the inverse implication is not true: it is known \cite[Prop. 2.4.3, Sect. 11.2.7]{lodayvallette} that there are (strict) quasi-isomorphisms of $\C$-coalgebras that are not sent to quasi-isomorphisms under the cobar construction.

\begin{proof}
	The proof is similar to \cite[Prop. 11.4.7]{lodayvallette}. Suppose $\Phi:D'\rightsquigarrow D$ is an $\alpha$-weak equivalence of $\C$-coalgebras. We have the commutative diagram
	\begin{center}
		\begin{tikzpicture}
			\node (a) at (0,1.5){$\bar_\alpha\Omega_\alpha D'$};
			\node (b) at (3,1.5){$\bar_\alpha\Omega_\alpha D$};
			\node (c) at (0,0){$D'$};
			\node (d) at (3,0){$D$};
			
			\draw[->] (a) -- node[above]{$\bar_\alpha\Phi$}node[below]{$\sim$} (b);
			\draw[->,line join=round,decorate,decoration={zigzag,segment length=4,amplitude=.9,post=lineto,post length=2pt}] (a) -- node[left]{$\varepsilon_{D'}$} (c);
			\draw[->,line join=round,decorate,decoration={zigzag,segment length=4,amplitude=.9,post=lineto,post length=2pt}] (b) -- node[right]{$\varepsilon_{D}$} (d);
			\draw[->,line join=round,decorate,decoration={zigzag,segment length=4,amplitude=.9,post=lineto,post length=2pt}] (c) -- node[above]{$\Phi$} (d);
		\end{tikzpicture}
	\end{center}
	where the vertical arrows are $\infty_\alpha$-quasi-isomorphisms by \cref{prop:rectification is oo-qi to original} and the top arrow is a quasi-isomorphism by \cite[Prop. 11.2.3]{lodayvallette}. Restriction to the first component gives us the diagram
	\begin{center}
		\begin{tikzpicture}
			\node (a) at (0,1.5){$\bar_\alpha\Omega_\alpha D'$};
			\node (b) at (3,1.5){$\bar_\alpha\Omega_\alpha D$};
			\node (c) at (0,0){$D'$};
			\node (d) at (3,0){$D$};
			
			\draw[->] (a) -- node[above]{$\bar_\alpha\Phi$}node[below]{$\sim$} (b);
			\draw[->] (a) -- node[left]{$\varepsilon_{D'}$} node[above,sloped]{$\sim$} (c);
			\draw[->] (b) -- node[right]{$\varepsilon_{D}$} node[below,sloped]{$\sim$} (d);
			\draw[->] (c) -- node[above]{$\Phi$} (d);
		\end{tikzpicture}
	\end{center}
	from which the statement follows.
\end{proof}

Finally, we can show that $\alpha$-weak equivalences of coalgebras are equivalent to zig-zags of weak equivalences of coalgebras, which is analogous to \cite[Thm. 11.4.9]{lodayvallette}.

\begin{proposition}\label{prop:alpha-we and zig-zags}
	Let $C$ and $D$ be two $\C$-coalgebras. The following are equivalent.
	\begin{enumerate}
		\item \label{pt:1} There is a zig-zag of weak equivalences
		\[
		C\longrightarrow\bullet\longleftarrow\cdots\longrightarrow\bullet\longleftarrow D\ .
		\]
		\item \label{pt:2} There are two weak equivalences forming a zig-zag
		\[
		C\longrightarrow\bullet\longleftarrow D\ .
		\]
		\item \label{pt:3} There is an $\alpha$-weak equivalence
		\[
		C\rightsquigarrow D\ .
		\]
	\end{enumerate}
\end{proposition}

\begin{proof}
	The fact that (\ref{pt:2}) implies (\ref{pt:1}) is obvious.
	
	\medskip
	
	We prove that (\ref{pt:3}) implies (\ref{pt:2}). Suppose we have an $\alpha$-weak equivalence
	\[
	\Phi:C\rightsquigarrow D\ .
	\]
	Then $\bar_\alpha\Phi$ is a weak equivalence, and thus we have the zig-zag
	\[
	C\xrightarrow{\eta_C}\bar_\alpha\Omega_\alpha C\xrightarrow{\bar_\alpha\Phi}\bar_\alpha\Omega_\alpha D\xleftarrow{\eta_D}D\ ,
	\]
	where the counits of the bar-cobar adjunction are weak equivalences by \cite[Thm. 2.6(2)]{val14}.
	
	\medskip
	
	Finally, we show that (\ref{pt:1}) implies (\ref{pt:3}). Every weak equivalence of coalgebras is in particular an $\alpha$-weak equivalence. Therefore, it is enough to prove that whenever we have a weak equivalence
	\[
	C\stackrel{\phi}{\longleftarrow}D
	\]
	of coalgebras, then we have an $\alpha$-weak equivalence going the other way round. Since $\phi$ is a weak equivalence, we have that
	\[
	\Omega_\alpha C\xleftarrow{\Omega_\alpha\phi}\Omega_\alpha D
	\]
	is a quasi-isomorphism. Moreover, every $\P$-algebra is fibrant and $\Omega_\alpha$ lands in the cofibrant $\P$-algebras by \cite[Thm. 2.9(1)]{val14}. Therefore, we can apply \cite[Lemma 4.24]{ds95} to obtain a homotopy inverse
	\[
	\Omega_\alpha C\stackrel{\Psi}{\longrightarrow}\Omega_\alpha D\ ,
	\]
	which is again a quasi-isomorphism, and thus defines an $\alpha$-weak equivalence from $C$ to $D$, as we desired.
\end{proof}

\section{Convolution homotopy Lie algebras} \label{sect:convolutionHomotopyAlg}

Both authors independently discovered that, given an operadic twisting morphism, one canonically obtains a morphism from the operad $\L_\infty$ to a certain convolution operad. This was done in full generality by the second author in \cite{w16} and in a slightly different particular case by the first author in \cite{rn17}. This produces a natural way to induce an $\L_\infty$-algebra structure on certain convolution algebras. In this section, we recall how this is done and give an interpretation of the Maurer--Cartan elements of the $\L_\infty$-algebras obtained this way.

\medskip

Let $\C$ be a cooperad, let $\P$ be an operad, and let
\[
\alpha:\C\longrightarrow\P
\]
be a twisting morphism from $\C$ to $\P$.

We will denote by $\alpha(n)$ the restriction of $\alpha$ to $\C(n)$. We define a morphism of operads
\[
\M_\alpha:\L_\infty\longrightarrow\hom(\C,\P)
\]
from the operad $\L_\infty$ to the convolution operad $\hom(\C,\P)$ by setting
\[
\M_\alpha(\ell_n)\in\hom(\C(n),\P(n))
\]
to be $\alpha(n)$.

\begin{theorem} \label{thm:bijTwHoms}
	Let $\C$ be a cooperad and let $\P$ be an operad. There is a bijection
	\[
	\Tw(\C,\P)\stackrel{\cong}{\longrightarrow}\hom_{\op}(\L_\infty,\hom(\C,\P))\ .
	\]
	The image of a twisting morphism $\alpha$ is given by the morphism of operads $\M_\alpha$ described above. This assignment is natural in the sense that if $f:\C'\longrightarrow\C$ is a morphism of cooperads and $g:\P\longrightarrow\P'$ is a morphism of operads, then
	\[
	\M_{f^*\alpha} = (f^*)\M_\alpha\qquad\text{and}\qquad\M_{g_*\alpha}=(g_*)\M_\alpha\ .
	\]
	Here, $(f_*)$ denotes the pullback
	\[
	f^*:\hom(\C,\P)\longrightarrow\hom(\C',\P)\ ,
	\]
	and similarly $(g_*)$ denotes the pushforward by $g$.
\end{theorem}

\begin{proof}
	We repeat here the argument presented in \cite{w16}. A twisting morphism is nothing else than a Maurer--Cartan element in the Lie algebra associated to the convolution operad $\hom(\C,\P)$. We have a canonical isomorphism of operads
	\[
	\hom(\C,\P)\cong\hom(\com^\vee,\hom(\C,\P))\ ,
	\]
	so that a Maurer--Cartan element in $\hom(\C,\P)$ is the same thing as a twisting morphism in
	\[
	\Tw(\com^\vee,\hom(\C,\P))\cong\hom_\op(\Omega\com^\vee,\hom(\C,\P))\ .
	\]
	This gives the desired bijection. Checking that the image of a twisting morphism $\alpha$ is given by the morphism of operads $\M_\alpha$ is simply a matter of unwinding definitions, and the proof of naturality is left as an easy exercise to the interested reader.
\end{proof}

\begin{remark}
	If we work with $\mathscr{L}_\infty$-algebras instead of $\L_\infty$-algebras, the result still holds but suspensions appear. The bijection of \cref{thm:bijTwHoms} in this case becomes
	\[
	\Tw(\C,\P)\stackrel{\cong}{\longrightarrow}\hom_\op(\mathscr{L}_\infty,\hom(\susp^c\otimes\C,\P))\ .
	\]
	All further results of the paper hold in this context, but one must be careful not to forget suspensions and signs that might appear.
\end{remark}

\begin{remark} \label{rem:SHLpaperRN}
	In \cite{rn17}, the following special case was treated. Let
	\[
	\Psi:\Q\longrightarrow\P
	\]
	be a morphism of operads. Then one obtains a canonical morphism
	\[
	\L_\infty\longrightarrow\hom(\bar\Q,\P)\ .
	\]
	This is recovered from \cref{thm:bijTwHoms} simply by considering the twisting morphism
	\[
	\bar\Q\stackrel{\pi}{\longrightarrow}\Q\stackrel{\Psi}{\longrightarrow}\P\ .
	\]
\end{remark}

Let $D$ be a conilpotent $\C$-coalgebra, and let $A$ be a $\P$-algebra. Then the chain complex $\hom(D,A)$ is naturally a $\hom(\C,\P)$-algebra, and thus a $\L_\infty$-algebra by pullback of the structure along the map $\M_\alpha$.

\begin{definition}\label{def:homAlpha}
	We denote by $\hom^\alpha(D,A)$, the chain complex $\hom(D,A)$ with the induced $\L_\infty$-algebra structure coming from $\alpha$. We  call this the \emph{convolution $\L_\infty$-algebra} of $D$ and $A$ (relative to $\alpha$).
\end{definition}

The compatibility with compositions stated in \cref{thm:bijTwHoms} translates as follows at the level of algebras.

\begin{lemma}\label{lemma:naturality of convolution algebras wrt compositions}
	Let $\alpha\in\Tw(\C,\P)$ be a twisting morphism, and suppose $f:\C'\to\C$ is a morphism of cooperads, and that $g:\P\to\P'$ is a morphism of operads.
	\begin{enumerate}
		\item\label{pt1} Let $A$ be a $\P$-algebra, let $D$ be a conilpotent $\C'$-coalgebra. We have
		\[
		\hom^{f^*\alpha}(D,A) = \hom^\alpha(f_*D,A)
		\]
		as $\L_\infty$-algebras.
		\item Dually, let $D$ be a $\C$-coalgebra, let $A$ be a $\P'$-algebra. We have
		\[
		\hom^{g_*\alpha}(D,A) = \hom^\alpha(D,g^*A)
		\]
		as $\L_\infty$-algebras.
	\end{enumerate}
\end{lemma}

\begin{proof}
	We only prove (\ref{pt1}), the other case being dual. Let $f_1,\ldots,f_n\in\hom(D,A)$, and denote by $F\coloneqq f_1\otimes\cdots\otimes f_n$. We have
	\begin{align*}
		\gamma_{\hom^{f^*\alpha}(D,A)}(\mu_n^\vee\otimes_{\S_n}F) =&\ \gamma_{\hom(D,A)}(\M_{f^*\alpha}(\mu_n^\vee)\otimes F)\\
		=&\ \gamma_{\hom(D,A)}((f^*)\M_\alpha(\mu_n^\vee)\otimes F)\\
		=&\ \gamma_{\hom(f_*D,A)}(\M_\alpha(\mu_n^\vee)\otimes F)\\
		=&\ \gamma_{\hom^\alpha(f_*D,A)}(\mu_n^\vee\otimes_{\S_n}F)\ ,
	\end{align*}
	where $\hom(D,A)$ is seen as a $\hom(\C',\P)$-algebra and $\hom(f_*D,A)$ is seen as a $\hom(\C,\P)$-algebra.
\end{proof}

Given an $\L_\infty$-algebra, it is natural to wonder what its Maurer--Cartan elements represent.

\begin{theorem}\label{thm:MCelOfConvAlg}
	With notations as above, we have that
	\[
	\MC(\hom^\alpha(D,A))\cong\Tw_\alpha(D,A)\ .
	\]
\end{theorem}

This result was also independently proven in both \cite{w16} and \cite{rn17}. The special case where $\P$ is a binary quadratic algebra, $\C=\P^{\antishriek}$, and $\kappa$ is the canonical twisting morphism was already known in the literature, see e.g. \cite[Sect. 11.1.2]{lodayvallette}. In that case, the convolution $\L_\infty$-algebra is in fact a strict Lie algebra and everything is much simpler. Notice that, thanks to \cref{thm:RosettaStone}, this result shows that Maurer--Cartan elements in the convolution $\L_\infty$-algebra can be interpreted as either morphisms of $\P$-algebras, or as morphisms of $\C$-coalgebras.

\section{Functoriality with respect to \texorpdfstring{$\infty$}{oo}-morphisms} \label{sect:functoriality}

In the previous section we defined the $\L_{\infty}$-convolution algebra. It is clear that it is functorial with respect to strict morphisms of algebras and coalgebras in both slots. In this section, we will prove that it is possible to extend the bifunctor $\hom^\alpha$ so that it accepts $\infty_\alpha$-morphisms in either one of its two slots. However, in \cref{sect:counterexample} we will prove that it can not be extended to take $\infty_\alpha$-morphisms in both slots at the same time.

\subsection{Construction of the induced infinity-morphisms}

As just remarked, it is straightforward to see that $\hom^\alpha$ behaves well with respect to strict morphisms in both slots, and thus defines a bifunctor
\[
\hom^\alpha:(\mathsf{conil.}\ \Ccog)^{op}\times\Palg\longrightarrow\Lalg\ .
\]
It is natural to ask what is its behavior with respect to $\infty_{\alpha}$-morphisms.

\medskip

Let $D,D'$ be two $\C$-coalgebras and let $A$ be a $\P$-algebra. Given an $\infty_\alpha$-morphism
\[
\Phi:\Omega_\alpha D'\longrightarrow\Omega_\alpha D
\]
of $\C$-coalgebras, we define a morphism of cocommutative coalgebras over gra\-ded vector spaces (i.e. not commuting with the differentials \emph{a priori})
\[
\hom^\alpha_\ell(\Phi,1):\bar_\iota\hom^\alpha(D,A)\longrightarrow\bar_\iota\hom^\alpha(D',A)
\]
as follows. First recall that such a morphism is completely determined by its post-composition with the projection
\[
\bar_\iota\hom^\alpha(D',A)\longrightarrow\hom^\alpha(D',A)\ .
\]
So let $f_1,\ldots,f_n\in\hom(D,A)$ and denote $F\coloneqq f_1\otimes\cdots\otimes f_n$. Then $\hom^\alpha_\ell(\Phi,1)(\mu_n^\vee\otimes F)$ is given by the composite
\begin{center}
	\begin{tikzpicture}
		\node (a) at (0,3){$D'$};
		\node (b) at (2.5,3){$\P(D)$};
		\node (c) at (0,0){$A$};
		\node (d) at (2.5,0){$\P(A)$};
		\node (e) at (2.5,1.5){$\P(n)\otimes_{\S_n}D^{\otimes n}$};
		
		\draw[->] (a) -- node[above]{$\Phi$} (b);
		\draw[->] (b) -- node[right]{$\proj_n$} (e);
		\draw[->] (e) -- node[right]{$F$} (d);
		\draw[->] (d) -- node[above]{$\gamma_A$} (c);
		\draw[dashed,->] (a) -- (c);
	\end{tikzpicture}
\end{center}
where
\[
\proj_n:\P(D)\longrightarrow\P(n)\otimes_{\S_n}D^{\otimes n}
\]
is the projection and $F$ acts on $\P(n)\otimes_{\S_n}D^{\otimes n}$ by sending a coinvariant $p\otimes x_1\otimes\cdots \otimes x_n$ to
\begin{equation}\label{eq:actionOfF}
	\sum_{\sigma\in\S_n}(-1)^\theta p\otimes f_{\sigma(1)}(x_1)\otimes\cdots\otimes f_{\sigma(n)}(x_n)\in\P(n)\otimes_{\S_n}A^{\otimes n}\ .
\end{equation}
Here, $\theta$ is the sign obtained by the Koszul sign rule, i.e.
\[
\theta = |p||F| + \sigma(F) + \sum_{i=1}^n|x_i|\left(\sum_{j=i+1}^n|f_{\sigma(j)}|\right)\ .
\]
Dually, let $D$ be a conilpotent $\C$-coalgebra and let $A,A'$ be two $\P$-algebras. Given an $\infty_\alpha$-morphism
\[
\Psi:\bar_\alpha A\longrightarrow\bar_\alpha A'
\]
of $\P$-algebras, we define a morphism of cocommutative coalgebras over graded vector spaces
\[
\hom^\alpha_r(1,\Psi):\bar_\iota\hom^\alpha(D,A)\longrightarrow\bar_\iota\hom^\alpha(D,A')
\]
by giving $\hom^\alpha_r(1,\Psi)(\mu_n^\vee\otimes F)$ as the following composite.
\begin{center}
	\begin{tikzpicture}
		\node (a) at (0,3){$D$};
		\node (b) at (2.5,3){$\C(D)$};
		\node (c) at (0,0){$A$};
		\node (d) at (2.5,0){$\C(A)$};
		\node (e) at (2.5,1.5){$\C(n)\otimes_{\S_n}D^{\otimes n}$};
		
		\draw[->] (a) -- node[above]{$\Delta_D$} (b);
		\draw[->] (b) -- node[right]{$\proj_n$} (e);
		\draw[->] (e) -- node[right]{$F$} (d);
		\draw[->] (d) -- node[above]{$\Psi$} (c);

		\draw[dashed,->] (a) -- (c);
	\end{tikzpicture}
\end{center}
Notice that here we crucially used the fact that we are working over a field of characteristic $0$ in order to use coinvariants instead of invariants.

\medskip

Using the notation
\[
\Delta_D^n\coloneqq\proj_n\Delta_D\ ,
\]
we can write
\[
\hom^\alpha_\ell(\Phi,1)(\mu_n^\vee\otimes F) = \gamma_AF\phi_n
\]
and
\[
\hom^\alpha_r(1,\Psi)(\mu_n^\vee\otimes F) = \Psi F\Delta_D^n\ .
\]

\subsection{Extension of the bifunctor to infinity-morphisms}

The main result for the maps we just constructed is the following one.

\begin{theorem} \label{thm:twoBifunctors}
	Let the various objects be as before.
	\begin{itemize}
		\item[a)] The map $\hom^\alpha_\ell(\Phi,1)$ is an $\infty$-morphism of $\L_\infty$-algebras.
		\item[b)] The map $\hom^\alpha_r(1,\Psi)$ is an $\infty$-morphism of $\L_\infty$-algebras.
	\end{itemize}
\end{theorem}

\begin{remark}
	The point (b) of \cref{thm:twoBifunctors} was already proven in \cite[Prop. 4.4]{rn17} in a slightly different setting, namely for tensor products instead of hom spaces and with twisting morphisms of the form
	\[
	\alpha:\bar\Q\xrightarrow{\ \pi\ }\Q\xrightarrow{\ \Psi\ }\P\ ,
	\]
	where $\Psi$ is a morphism of operads and $\Q$ is Koszul.
\end{remark}

\begin{proof}
	We begin by proving (a). We see $\hom^\alpha(\Phi,1)$ as a linear map
	\[
	\hom^\alpha_\ell(\Phi,1):\bar_\iota\hom^\alpha(D,A)\longrightarrow\hom^\alpha(D',A')\ ,
	\]
	so that our goal becomes to prove that it satisfies the Maurer--Cartan equation
	\begin{equation}\label{eq:MCforHom(Phi,1)}
		\partial(\hom^\alpha_\ell(\Phi,1)) + \star_\iota(\hom^\alpha_\ell(\Phi,1)) = 0\ .
	\end{equation}
	This goes roughly as follows. We have
	\[
	\partial(\hom^\alpha_\ell(\Phi,1)) = d_{\hom(D',A)}\hom^\alpha_\ell(\Phi,1) - \hom^\alpha_\ell(\Phi,1)d_{\bar_\iota\hom^\alpha(D,A)}\ .
	\]
	With some algebraic and combinatorial manipulations, as well as using the fact that $\Phi$ is an $\infty_\alpha$-morphism, we prove that
	\[
	\hom^\alpha_\ell(\Phi,1)d_{\bar_\iota\hom^\alpha(D,A)} = d_A\hom^\alpha_\ell(\Phi,1) + (\bigstar)\ ,
	\]
	where $(\bigstar)$ is a term involving $\star_\alpha(\Phi)$. Then, again through algebraic manipulations, $(\bigstar)$ is shown to be equal to $\star_\iota(\hom^\alpha_\ell(\Phi,1))$, completing the proof.
	
	\medskip
	
	We do all of this by checking it holds on elements of $\bar_\iota(\hom^\alpha(D,A))$. Let
	\[
	f_1,\ldots,f_n\in\hom(D,A)
	\]
	and again denote $F\coloneqq f_1\otimes\cdots\otimes f_n$. If $S=\{i_1<\ldots<i_k\}\subseteq[n]$, we use the notation $F^S\coloneqq f^{i_1}\otimes\cdots\otimes f^{i_k}$. We have
	\begin{align*}
		d_{\bar_\iota\hom^\alpha(D,A)}(\mu_n^\vee\otimes F) =&\ d_1(\mu_n^\vee\otimes F) + d_2(\mu_n^\vee\otimes F)\\
		=&\ \mu_n^\vee\otimes\left(d_{A^{\otimes n}}F - (-1)^{|F|}Fd_{D^{\otimes n}}\right) +\\
		&+ \sum_{S_1\sqcup S_2 = [n]}(-1)^\epsilon\mu_{n_1}^\vee\otimes\left(\ell_{n_2}(F^{S_2})\otimes F^{S_1}\right)\ ,
	\end{align*}
	where $n_1 = |S_1|+1$, and $\epsilon = \sh(S_2,S_1)$ is the sign appearing by Koszul sign rule, as defined in \cref{subsect:KoszulSignRule}. We also have
	\[
	\ell_{n_2}(F^{S_2}) = \gamma_A(\alpha\circ 1_A)F^{S_2}\Delta_D^{n_2}\ ,
	\]
	where $n_2 = |S_2|$. Now we compute
	\begin{align*}
		\hom^\alpha_\ell(\Phi,1)&(d_{\bar_\iota\hom^\alpha(D,A)}(\mu_n^\vee\otimes F)) =\\
		=&\ \hom^\alpha_\ell(\Phi,1)\left(\mu_n^\vee\otimes\left(d_{A^{\otimes n}}F - (-1)^{|F|}Fd_{D^{\otimes n}}\right)\right) +\\
		&+ \hom^\alpha_\ell(\Phi,1)\left(\sum_{S_1\sqcup S_2 = [n]}(-1)^\epsilon\mu_{n_1}^\vee\otimes\left(\ell_{n_2}(F^{S_2})\otimes F^{S_1}\right)\right)\\
		=&\ \gamma_A(1_\P\circ'd_A)F\phi_n - (-1)^{|F|}\gamma_AF(1_\P\circ'd_D)\phi_n +\\
		&+ \sum_{S_1\sqcup S_2 = [n]}(-1)^\epsilon\gamma_A\left(\ell_{n_2}(F^{S_2})\otimes F^{S_1}\right)\phi_{n_1}\ .
	\end{align*}
	The first term simply gives
	\begin{align*}
	\gamma_A(1_\P\circ'd_A)F\phi_n =&\ d_A\gamma_AF\phi_n - \gamma_A(d_\P\circ1_A)F\phi_n\\
	=&\ d_A\hom^\alpha_\ell(\Phi,1)(\mu_n^\vee\otimes F) - \gamma_A(d_\P\circ1_A)F\phi_n\ .
	\end{align*}
	For the third term, we have
	\begin{align*}
		\sum_{S_1\sqcup S_2 = [n]}&(-1)^\epsilon\gamma_A\left(\ell_{n_2}(F^{S_2})\otimes F^{S_1}\right)\phi_{n_1} =\\
		=&\ \sum_{S_1\sqcup S_2 = [n]}(-1)^\epsilon\gamma_A\left(\left(\gamma_A(\alpha\circ 1_A)F^{S_2}\Delta_D^{n_2}\right)\otimes F^{S_1}\right)\phi_{n_1}\\
		=&\ \gamma_A(1_\P\circ_{(1)}\gamma_A)(1_\P\circ_{(1)}(\alpha\circ1_A))\sum_{S_1\sqcup S_2 = [n]}(-1)^\epsilon\left(\left(F^{S_2}\Delta_D^{n_2}\right)\otimes F^{S_1}\right)\phi_{n_1}\\
		=&\ \gamma_A(\gamma_\P\circ1_A)(1_\P\circ_{(1)}(\alpha\circ1_A))\sum_{S_1\sqcup S_2 = [n]}(-1)^\epsilon\left(\left(F^{S_2}\Delta_D^{n_2}\right)\otimes F^{S_1}\right)\phi_{n_1}\ .
	\end{align*}
	Now we study the action of
	\[
	\sum_{S_1\sqcup S_2 = [n]}(-1)^\epsilon\left(\left(F^{S_2}\Delta_D^{n_2}\right)\otimes F^{S_1}\right)
	\]
	on a fixed element $p\otimes x_1\otimes\cdots\otimes x_{n_1}\in\P(n_1)\otimes_{\S_{n_1}}D^{\otimes n_1}$. Using Sweedler-like notation, we write
	\[
	\Delta_D^{n_2}(x_j) = c_j\otimes y_1^j\otimes\cdots\otimes y_{n_2}^j\in\C(n_2)\otimes_{\S_{n_2}}D^{\otimes n_2}\ .
	\] 
	
	We fix $n_1,n_2$ such that $n_1+n_2 = n+1$ and we restrict to partitions $S_1\sqcup S_2 = [n]$ such that $|S_1|=n_1-1$ and $|S_2|=n_2$. Then we have	
	\begin{align*}
		&\sum_{\substack{S_1\sqcup S_2 = [n]\\|S_1|=n_1-1,|S_2|=n_2}}(-1)^\epsilon\left(\left(F^{S_2}\Delta_D^{n_2}\right)\otimes F^{S_1}\right)(p\otimes x_1\otimes\cdots\otimes x_{n_1}) =\\
		&\\
		&\ = \sum_{j=1}^{n_1}\sum_{\substack{S_1\sqcup S_2 = [n]\\|S_1|=n_1-1,|S_2|=n_2\\\tau_1\in\S_{n_1-1}}}(-1)^{\epsilon + \epsilon'}p\otimes f_{\tau_1(i_1)}(x_1)\otimes\cdots\otimes f_{\tau_1(i_{j-1})}(x_{j-1})\otimes\\
		&\qquad\qquad\qquad\qquad\qquad\otimes(F^{S_2}\Delta_D^{n_2})(x_j)\otimes f_{\tau_1(i_{j+1})}(x_{j+1})\otimes\cdots\otimes f_{\tau_1(i_{n_1})}(x_{n_1})\\
		&\\
		&\ = \sum_{j=1}^{n_1}\sum_{\substack{S_1\sqcup S_2 = [n]\\|S_1|=n_1-1,|S_2|=n_2\\\tau_1\in\S_{n_1-1}}}(-1)^{\epsilon + \epsilon'+\epsilon''}(p\otimes_jc_j)\otimes f_{\tau_1(i_1)}(x_1)\otimes\cdots\otimes f_{\tau_1(i_{j-1})}(x_{j-1})\otimes\\
		&\qquad\qquad\qquad\qquad\qquad\otimes\sum_{\tau_2\in\S_{n_2}}f_{\tau_2(k_1)}(y_1^j)\otimes\cdots\otimes f_{\tau_2(k_{n_2})}(y_{n_2}^j)\otimes\\
		&\qquad\qquad\qquad\qquad\qquad\otimes f_{\tau_1(i_{j+1})}(x_{j+1})\otimes\cdots\otimes f_{\tau_1(i_{n_1})}(x_{n_1})\ ,		
	\end{align*}
	where $S_1 = \{i_1<i_2<\ldots<i_{j-1}<i_{j+1}<\ldots<i_{n_1}\}$ and $S_2 = \{k_1<\ldots<k_{n_2}\}$. The signs $\epsilon'$ and $\epsilon''$ are given by the Koszul sign rule, explicitly
	\[
	\epsilon' = |p||F|+\tau_1(F^{S_1})+\sum_{\substack{1\le a\le n_1\\a\neq j}}|f_{\tau_1(i_a)}|\sum_{b=1}^{i_a-1}|x_b| + |F^{S_2}|\sum_{c=1}^{j-1}(|f_{\tau_1(i_c)}|+|x_c|)\ ,
	\]
	and
	\[
	\epsilon'' = |c_j||F^{S_2}| + \tau_2(F^{S_2}) + \sum_{a=1}^{n_2}|f_{\tau_2(k_a)}|\sum_{b=1}^{k_a-1}|y_b^j| + |c_j|\sum_{c=1}^{j-1}(|f_{\tau_1(i_c)}|+|x_c|)\ .
	\]
	Fixing $j$, we have an obvious bijection
	\[
	\S_n\longrightarrow\{S_1\sqcup S_2 = [n]\mid|S_1|=n_1-1,|S_2|=n_2\}\times\S_{n_1-1}\times\S_{n_2}\ .
	\]
	Applying this, we obtain that the above equals
	\begin{align*}
		&\sum_{j=1}^{n_1}\sum_{\tau\in\S_n}(-1)^{\theta+|c_j|\sum_{c=1}^{j-1}|x_c|}(p\otimes_jc_j)\otimes f_{\tau(1)}(x_1)\otimes\cdots\otimes f_{\tau(j-1)}(x_{j-1})\otimes\\
		&\qquad\qquad\qquad\otimes f_{\tau(j)}(y_1^j)\otimes\cdots\otimes f_{\tau(j+n_2)}(y_{n_2}^j)\otimes\\
		&\qquad\qquad\qquad\otimes f_{\tau(j+n_2+1)}(x_{j+1})\otimes\cdots\otimes f_{\tau(n)}(x_{n_1})\\
		&\qquad = F(1_\P\circ'\Delta_D^{n_2})(p\otimes x_1\otimes\cdots\otimes x_{n_1})\ ,
	\end{align*}
	where $\theta$ is the sign of \cref{eq:actionOfF} and was recovered by noticing that
	\[
	\sh(S_2,S_1) + \tau_1(F^{S_1}) + \tau_2(F^{S_2}) + |F^{S_2}|\sum_{c=1}^{j-1}|x_c| = \tau(F)
	\]
	under the correspondence above, together with some careful sign manipulations.	It follows that
	\begin{align*}
		\sum_{S_1\sqcup S_2 = [n]}(-1)^\epsilon\gamma_A&\left(\ell_{n_2}(F^{S_2})\otimes F^{S_1}\right)\phi_{n_1} =\\
		=&\ \gamma_A(\gamma_\P\circ1_A)(1_\P\circ_{(1)}(\alpha\circ1_A))\sum_{S_1\sqcup S_2 = [n]}F(1\circ_{(1)}\Delta_D^{n_2})\phi_{n_1}\\
		=&\ (-1)^{|F|}\gamma_AF\proj_n(\gamma_\P\circ_{(1)}1_A)(1_\P\circ_{(1)}(\alpha\circ1_A))(1_\P\circ'\Delta_D)\Phi\\
		=&\ -(-1)^{|F|}\gamma_AF\proj_nd^{\Omega_\alpha D}_2\Phi\ .
	\end{align*}
	Therefore, we have
	\begin{align*}
		\hom^\alpha_\ell(\Phi,1)&(d_{\bar_\iota\hom^\alpha(D,A)}(\mu_n^\vee\otimes F)) =\\
		=&\ d_A\hom^\alpha_\ell(\Phi,1)(\mu_n^\vee\otimes F) - \gamma_A(d_\P\circ1_A)F\phi_n\\
		&\qquad - (-1)^{|F|}\gamma_AF(1_\P\circ'd_D)\phi_n - (-1)^{|F|}\gamma_AFd_2^{\Omega_\alpha D}\phi_n\\	
		=&\ d_A\hom^\alpha_\ell(\Phi,1)(\mu_n^\vee\otimes F) - (-1)^{|F|}\gamma_AF(d_\P\circ1_D)\phi_n\\
		&\qquad - (-1)^{|F|}\gamma_AF(1_\P\circ'd_D)\phi_n - (-1)^{|F|}\gamma_AFd_2^{\Omega_\alpha D}\phi_n\\
		=&\ d_A\hom^\alpha_\ell(\Phi,1)(\mu_n^\vee\otimes F) - (-1)^{|F|}\gamma_AF\proj_nd_{\Omega_\alpha D}\Phi\\
		=&\ d_A\hom^\alpha_\ell(\Phi,1)(\mu_n^\vee\otimes F) - (-1)^{|F|}\gamma_AF\proj_n\left(\Phi d_{D'} - \star_\alpha(\Phi)\right)\\
		=& \left(d_{\hom(D',A)}\hom^\alpha_\ell(\Phi,1)\right)(\mu_n^\vee\otimes F) + (-1)^{|F|}\gamma_AF\proj_n\star_\alpha(\Phi)\ ,
	\end{align*}
	where in the fourth equality we used the Maurer--Cartan equation for $\Phi$. Written otherwise, this gives
	\[
	0 = \partial(\hom^\alpha_\ell(\Phi,1))(\mu_n^\vee\otimes F) + (-1)^{|F|}\gamma_AF\proj_n\star_\alpha(\Phi)\ .
	\]
	To conclude, we only need to prove that
	\[
	(-1)^{|F|}\gamma_AF\proj_n\star_\alpha(\Phi) = \star_\iota(\hom^\alpha_\ell(\Phi,1))(\mu_n^\vee\otimes F)\ .
	\]
	Recall that $\star_\alpha(\Phi)$ is given by the composition
	\[
	\star_\alpha(\Phi) = \gamma_{\Omega_\alpha D}(\alpha\circ\Phi)\Delta_{D'} = (\gamma_\P\circ1_D)(\alpha\circ\Phi)\Delta_{D'}\ ,
	\]
	so that
	\[
	\gamma_AF\proj_n\star_\alpha(\Phi) = \gamma_A(\gamma_\P\circ1_D)F\proj_n(\alpha\circ\Phi)\Delta_{D'}\ .
	\]
	Here, we noticed that $(\gamma_\P\circ1_D)$ and $F\proj_n$ commute. Similarly, we have
	\begin{align*}
		\star_\iota&(\hom^\alpha_\ell(\Phi,1))(\mu_n^\vee\otimes F) =\\
		=&\ \gamma_{\hom^\alpha(D',A)}(\iota\circ\hom^\alpha_\ell(\Phi,1))\Delta_{\bar_\iota\hom^\alpha(D',A)}(\mu_n^\vee\otimes F)\\
		=&\ \medmath{\gamma_{\hom^\alpha(D',A)}(\iota\circ\hom^\alpha_\ell(\Phi,1))\left(\sum_{\substack{S_1\sqcup\ldots\sqcup S_k = [n]}}(-1)^\kappa\mu_k^\vee\otimes(\mu_{n_1}^\vee\otimes F^{S_1})\otimes\cdots\otimes(\mu_{n_k}^\vee\otimes F^{S_k})\right)}\\
		=&\ \gamma_{\hom^\alpha(D',A)}\left(\sum_{\substack{S_1\sqcup\ldots\sqcup S_k = [n]}}(-1)^\kappa\mu_k^\vee\otimes(\gamma_AF^{S_1}\phi_{n_1})\otimes\cdots\otimes(\gamma_AF^{S_k}\phi_{n_k})\right)\\
		=&\ \sum_{\substack{S_1\sqcup\ldots\sqcup S_k = [n]}}(-1)^\kappa\gamma_A\big(\alpha\otimes(\gamma_AF^{S_1}\phi_{n_1})\otimes\cdots\otimes(\gamma_AF^{S_k}\phi_{n_k})\big)\Delta_{D'}^k\\
		=&\ \gamma_A(\alpha\circ1_A)(1_\C\circ\gamma_A)F\proj_n(1_\C\circ\Phi)\Delta_{D'}\\
		=&\ (-1)^{|F|}\gamma_A(\gamma_\P\circ1_A)F\proj_n(\alpha\circ\Phi)\Delta_{D'}\ ,
	\end{align*}
	where $\kappa = \sh(F^{S_1},\ldots,F^{S_k})$ comes from the Koszul sign rule. In the third line, we used the fact that $\Delta_{\bar_\iota\hom^\alpha(D,A)} = \Delta_{\com^\vee}\circ1_{\hom^\alpha(D,A)}$ on $\bar_\iota\hom^\alpha(D,A)=\com^\vee\circ\hom^\alpha(D,A)$ (as algebraic operads). In the fourth line, one must impose $|S_i|\ge2$ for every $i$, and $n_i\coloneqq|S_k|$. In the sixth line, we use an argument similar to the one we did above: fixing the cardinalities $n_i$ of the $S_i$, we have an isomorphism
	\[
	\S_n\longrightarrow\{S_1\sqcup\ldots\sqcup S_k\mid |S_i|=n_i\}\times\S_{n_1}\times\cdots\times\S_{n_k}\ .
	\]
	In the last line, the sign appears by Koszul sign rule because we switch $F$ and $\alpha$. It follows that
	\[
	\partial(\hom^\alpha_\ell(\Phi,1)) + \star_\iota(\hom^\alpha_\ell(\Phi,1)) = 0\ ,
	\]
	i.e. that $\hom^\alpha_\ell(\Phi,1)$ is an $\infty$-morphism of $\L_\infty$-algebras, as desired. This proves (a).
	
	\medskip
	
	The proof of (b) is completely analogous, but we sketch it here for completeness. As before, we start by computing
	\begin{align*}
		\hom^\alpha_r(1,\Psi)&(d_{\bar_\iota\hom^\alpha(D,A)}(\mu_n^\vee\otimes F)) =\\
		=&\ \Psi d_{A^\otimes n}F\Delta_D^n - (-1)^{|F|}\hom^\alpha_r(1,\Psi)(\mu_n^\vee\otimes F)d_D +\\
		&\ + \sum_{S_1\sqcup S_2 = [n]}(-1)^\epsilon\Psi\left(\ell_{n_2}(F^{S_2})\otimes F^{S_1}\right)\Delta_D^{n_1}\ .
	\end{align*}
	By the same kind of manipulations we did above, we get
	\begin{align*}
		\Psi d_{A^\otimes n}F\Delta_D^n + \sum_{S_1\sqcup S_2 = [n]}(-1)^\epsilon\Psi&\left(\ell_{n_2}(F^{S_2})\otimes F^{S_1}\right)\Delta_D^{n_1} =\\
		=&\ \Psi d_{\bar_\alpha A}F\Delta_D^n\\
		=&\ d_{A'}\Psi F\Delta_D^n + \star_\alpha(\Psi)F\Delta_D^n\ .
	\end{align*}
	It follows that
	\begin{align*}
	\hom^\alpha_r(1,\Psi)(d_{\bar_\iota\hom^\alpha(D,A)}&(\mu_n^\vee\otimes F)) =\\
	=&\ d_{\hom(A',D)}\hom^\alpha_r(1,\Psi)(\mu_n^\vee\otimes F) + \star_\alpha(\Psi)F\Delta_D^n\ ,
	\end{align*}
	and showing that
	\[
	\star_\alpha(\Psi)F\Delta_D^n = \star_\iota(\hom^\alpha_r(1,\Psi))(\mu_n^\vee\otimes F)
	\]
	gives the Maurer--Cartan equation
	\[
	\partial(\hom^\alpha_r(1,\Psi)) + \star_\iota(\hom^\alpha_r(1,\Psi)) = 0\ ,
	\]
	completing the proof.
\end{proof}

\begin{lemma}
	Let $\Phi_1,\Phi_2$ be composable $\infty_\alpha$-morphisms of $\C$-coalgebras, and let $\Psi_1,\Psi_2$ be composable $\infty_\alpha$-morphisms of $\P$-algebras. Then
	\begin{itemize}
		\item[a)] $\hom^\alpha_\ell(\Phi_1\Phi_2,1) =\hom^\alpha_\ell(\Phi_2,1)\hom^\alpha_\ell(\Phi_1,1)$, and
		\item[b)] $\hom^\alpha_r(1,\Psi_1\Psi_2) = \hom^\alpha_r(1,\Psi_1)\hom^\alpha_r(1,\Psi_2)$.
	\end{itemize}
\end{lemma}

\begin{proof}
	The full map
	\[
	\hom^\alpha_\ell(\Phi,1):\bar_\iota\hom^\alpha(D,A)\longrightarrow\bar_\iota\hom^\alpha(D',A)
	\]
	is given by
	\[
	\hom^\alpha(\Phi,1)(\mu_n^\vee\otimes F) = \sum_{S_1\sqcup\ldots\sqcup S_k = [n]}(-1)^\kappa\mu_k^\vee\otimes\left((\gamma_AF^{S_1}\phi_{n_1})\otimes\cdots\otimes(\gamma_AF^{S_k}\phi_{n_k})\right)\ ,
	\]
	where $\kappa = \sh(F^{S_1},\ldots,F^{S_k})$ comes from the Koszul sign rule, as usual. Thus, the projected map $\hom^\alpha(\Phi_1\Phi_2,1)$ is given by
	\begin{align*}
		\hom^\alpha(\Phi_1\Phi_2,1)&(\mu_n^\vee\otimes F) =\\
		=&\ \gamma_AF\proj_n\Phi_1\Phi_2\\
		=&\ \gamma_AF\sum_{n_1+\cdots+n_k = n}(\gamma_\P\circ1_D)(1_\P\otimes((\phi_1)_{n_1}\otimes\cdots\otimes(\phi_1)_{n_k}))(\phi_2)_k\\
		=&\ \gamma_A\sum_{S_1\sqcup\ldots\sqcup S_k = [n]}(-1)^\kappa\left(\gamma_AF^{S_1}(\phi_1)_{n_1}\otimes\cdots\otimes\gamma_AF^{S_k}(\phi_1)_{n_k}\right)(\phi_2)_k\\
		=&\ \hom^\alpha(\Phi_2,1)\hom^\alpha(\Phi_1,1)(\mu_n^\vee\otimes F)\ .
	\end{align*}
	In the last line, $\hom^\alpha(\Phi_1,1)$ must be interpreted as the full map. Once again, in the third line we used the fact that $\gamma_A(\gamma_\P\circ1_A) = \gamma_A(1_\P\circ\gamma_A)$ as well as the bijection
	\[
	\S_n\longrightarrow\{S_1\sqcup\ldots\sqcup S_k\mid|S_i|=n_1\}\times\S_{n_1}\times\cdots\times\S_{n_k}
	\]
	that we have whenever we fix $n_1+\cdots+n_k=n$. This proves point (a). The proof of (b) is completely analogous and is left as an exercise to the reader.
\end{proof}

\begin{corollary} \label{cor:twoBifunctors}
	The bifunctor
	\[
	\hom^\alpha:(\mathsf{conil.}\ \Ccog)^{op}\times\Palg\longrightarrow\Lalg
	\]
	extends to bifunctors
	\[
	\hom^\alpha_\ell:(\infty_\alpha\text{-}\Ccog)^{op}\times\Palg\longrightarrow\Lalg
	\]
	and
	\[
	\hom^\alpha_r:(\mathsf{conil.}\ \Ccog)^{op}\times\infty_\alpha\text{-}\Palg\longrightarrow\Lalg\ .
	\]
\end{corollary}

\begin{proof}
	The only thing left to check is the compatibility between $\infty_\alpha$-morphism in one slot and strict morphisms in the other one. This is straightforward and left as an exercise to the reader.
\end{proof}

The bifunctors we just defined are well behaved with respect to $\infty_\alpha$-quasi-iso\-morphisms.

\begin{proposition} \label{prop:qiInduceQi}
	For both versions of the bifunctor $\hom^\alpha(-,-)$ of \cref{cor:twoBifunctors}, if we fill one slot with an $\infty_\alpha$-quasi-isomorphism and the other one with a strict quasi-isomorphism, then the resulting $\infty$-morphism of $\L_\infty$-algebras is an $\infty$-quasi-isomor\-phism.
\end{proposition}

\begin{proof}
	The proof follows from the following facts. First, we have
	\[
	\hom^\alpha(\Phi,\psi)_1(f) = \psi f \phi_1
	\]
	for $\Phi$ an $\infty_\alpha$-morphism of $\C$-coalgebras and $\psi$ be a morphism of $\P$-algebras. Similarly, if $\Psi$ is an $\infty_\alpha$-morphism of $\P$-algebras and $\phi$ is a morphism of $\C$-coalgebras, then
	\[
	\hom^\alpha(\phi,\Psi)_1(f) = \psi_1 f \phi\ .
	\]
	Therefore, the statement reduces to proving that the pullback, resp. pushforward by a quasi-iso\-mor\-phism is a quasi-isomorphism. One can do this for example by noticing that, since we are working over a field, every chain complex is bifibrant, then applying \cite[Lemma 4.24]{ds95} to obtain a homotopy inverse to the quasi-isomorphism, and use it to prove that the pullback and pushforward are again quasi-isomorphisms.
\end{proof}

\section{Failure to be a bifunctor} \label{sect:counterexample}

The obvious thing one would try to do at this point is to define a bifunctor
\[
\hom^\alpha:(\infty_\alpha\text{-}\Ccog)^{op}\times\infty_\alpha\text{-}\Palg\longrightarrow\infty\text{-}\Lalg
\]
which restricts to the functors defined in the last subsection if we take a strict morphism in one of the two slots. Unfortunately this is not possible, as we will prove in this section.

\subsection{Introductory remarks}

We will work over a field of characteristic $0$ and in the non-symmetric setting (see \cref{subsect:nsCase}). If there were such a bifunctor, then we would necessarily have
\[
\hom^\alpha_\ell(\Phi,1)\hom^\alpha_r(1,\Psi) = \hom^\alpha(\Phi,\Psi) = \hom^\alpha_r(1,\Psi)\hom^\alpha_\ell(\Phi,1)
\]
for any couple of $\infty_\alpha$-morphisms. We will give an explicit example where this is not the case. For reference, notice that the two composites are given by the diagrams
\begin{center}
	\begin{tikzpicture}
		\node (a) at (.5,3){$D'$};
		\node (b) at (3,3){$\P(D)$};
		\node (c) at (6,3){$\P(\C(D))$};
		\node (d) at (.5,0){$A'$};
		\node (e) at (3,0){$\P(A')$};
		\node (f) at (6,0){$\P(\C(A))$};
		\node (g) at (6,1.5){$(\P\circ\C)(n)\otimes_{\S_n}D^{\otimes n}$};
		
		\draw[->] (a) -- node[above]{$\Phi$} (b);
		\draw[->] (b) -- node[above]{$\P(\Delta_D)$} (c);
		\draw[->] (g) -- node[right]{$F$} (f);
		\draw[->] (f) -- node[above]{$\P(\Psi)$} (e);
		\draw[->] (e) -- node[above]{$\gamma_{A'}$} (d);
		\draw[dashed,->] (a) -- node[left]{$\hom^\alpha_\ell(\Phi,1)\hom^\alpha_r(1,\Psi)(\mu_n^\vee\otimes F)$} (d);
		\draw[->] (c) -- node[right]{$\mathrm{proj}_n$} (g);
	\end{tikzpicture}
\end{center}
and
\begin{center}
	\begin{tikzpicture}
		\node (a) at (.5,3){$D'$};
		\node (b) at (3,3){$\C(D')$};
		\node (c) at (6,3){$\C(\P(D))$};
		\node (d) at (.5,0){$A'$};
		\node (e) at (3,0){$\C(A)$};
		\node (f) at (6,0){$\C(\P(A))$};
		\node (g) at (6,1.5){$(\C\circ\P)(n)\otimes_{\S_n}D^{\otimes n}$};
		
		\draw[->] (a) -- node[above]{$\Delta_{D'}$} (b);
		\draw[->] (b) -- node[above]{$\C(\Phi)$} (c);
		\draw[->] (g) -- node[right]{$F$} (f);
		\draw[->] (f) -- node[above]{$\C(\gamma_A)$} (e);
		\draw[->] (e) -- node[above]{$\Psi$} (d);
		\draw[dashed,->] (a) -- node[left]{$\hom^\alpha_r(1,\Psi)\hom^\alpha_\ell(\Phi,1)(\mu_n^\vee\otimes F)$} (d);
		\draw[->] (c) -- node[right]{$\mathrm{proj}_n$} (g);
	\end{tikzpicture}
\end{center}
respectively, when applied to $\mu_n^\vee\otimes F\in\bar_\iota(\hom(D,A))$.

\medskip

We will work with \emph{non-symmetric} associative algebras and (suspended) coassociative coalgebras. Since $\as(n)\cong\k$ for each $n\ge1$, for any associative algebra $A$ we will implicitly identify $\as(n)\otimes A^{\otimes n}$ with $A^{\otimes n}$ in some places, and similarly for coassociative coalgebras.

\subsection{The families \texorpdfstring{$A^n$}{An} and \texorpdfstring{$H^n$}{Hn}}

We define $A^n$ for $n\ge1$ as the commutative algebra
\[
A^n\coloneqq\overline{\k[x,y]}
\]
seen as an associative algebra. The overline means that we take the augmentation ideal of $\k[x,y]$, i.e. that we only consider polynomials with no constant term. The degrees are $|x|=0$ and $|y|=1$ and the differential is given by $dy = x^n$. Notice that $y^2=0$. We have
\[
d(x^a) = 0\ ,\qquad d(x^ay) = x^{a+n}\ .
\]
It follows that, as a chain complex,
\[
H^n\coloneqq H_\bullet(A^n) \cong \bigoplus_{a=1}^{n-1}\k z_a\ ,
\]
where $z_a=[x^a]$ is the class of $x^a$. We have three maps
\begin{center}
	\begin{tikzpicture}
		\node (a) at (0,0){$A_n$};
		\node (b) at (2,0){$H_n$};
		
		\draw[->] (a)++(.3,.1)--node[above]{\mbox{\tiny{$p$}}}+(1.4,0);
		\draw[<-,yshift=-1mm] (a)++(.3,-.1)--node[below]{\mbox{\tiny{$i$}}}+(1.4,0);
		\draw[->] (a) to [out=-150,in=150,looseness=4] node[left]{\mbox{\tiny{$h$}}} (a);
	\end{tikzpicture}
\end{center}
given by
\begin{enumerate}
	\item $i(z_a) = x^a$.
	\item $p(x^a) = y_a$ for $a<n$ and zero on all other monomials.
	\item $h(x^a) = x^{a-n}y$ for $a\ge n$ and zero on all other monomials.
\end{enumerate}

\begin{lemma}
	The maps described above form a contraction.
\end{lemma}

\begin{proof}
	This is a straightforward computation.
\end{proof}

Now we apply the Homotopy Transfer Theorem \cite[Thm. 1]{kad80}, \cite[Sect. 6.4]{ks00}, see also \cite[Sect. 9.4]{lodayvallette} for the a more modern treatment, to obtain an $\as_\infty$-algebra structure on $H^n$ and $\infty$-morphisms between the two algebras.

\begin{lemma}\label{lemma:An is formal}
	The algebra $A^n$ is formal, and
	\[
	H^n\cong\frac{\overline{\k[z]}}{(z^n)}
	\]
	as associative (and $\as_\infty$-) algebras.
\end{lemma}

\begin{proof}
	The arity $2$ operation in $H^n$ is given by
	\[
	m_2(z_a,z_b) = p(i(z_a)i(z_b)) = p(x^{a+b}) = \begin{cases}z_{a+b}&\text{if }a+b<n,\\0&\text{otherwise}.
	\end{cases}
	\]
	Therefore, the underlying associative algebra is indeed
	\[
	H^n\cong\frac{\overline{\k[z]}}{(z^n)}
	\]
	(keeping in mind that $d=0$ on $H^n$, so that associativity is indeed satisfied). For the higher operations, we notice that
	\[
	h(i(z_a)i(z_b)) = \begin{cases}x^{a+b-n}y&\text{if }a+b\ge n,\\0&\text{otherwise}.
	\end{cases}
	\]
	it follows that if we multiply by any element of $A_n$ and then apply either $h$ or $p$, we always get $0$. It follows that all higher operations are $0$, concluding the proof.
\end{proof}

\begin{lemma}
	The $\infty$-quasi-isomorphism $i_\infty$ of $\as_\infty$-algebras extending $i$ is given by $i_1=i$,
	\[
	i_2(z^a,z^b) = \begin{cases}x^{a+b-n}y&\text{if }a+b\ge n,\\0&\text{otherwise},
	\end{cases}
	\]
	and $i_n=0$ for all $n\ge3$.
\end{lemma}

\begin{proof}
	This is proven with computations analogous to the ones in the proof of \cref{lemma:An is formal}.
\end{proof}

Notice that, by \cref{lemma:equalityOfInftyMorphisms}, the $\infty$-morphism $i_\infty$ is an $\infty_\kappa$-morphism of associative algebras.

\subsection{A coalgebra and an \texorpdfstring{$\infty_\kappa$}{ook}-morphism}

A structure of conilpotent dg $\as^{\antishriek}$-coalgebra, that is a shifted coassociative coalgebra, on a graded vector space $V$ is the same thing as a square zero differential $d$ on $\as(V)$ such that
\[
d(V)\subseteq V\oplus V^{\otimes 2}.
\]
Let
\[
V\coloneqq\bigoplus_{i\ge1}\k v_i
\]
with $|v_i| = i$. We define
\[
d:\as(V)\longrightarrow\as(V)
\]
of degree $-1$ by
\[
d(v_i) = \sum_{j+k = i-1}(-1)^jv_j\otimes v_k\ .
\]

\begin{lemma}
	The map $d$ squares to $0$.
\end{lemma}

\begin{proof}
	This is a straightforward routine computation.
\end{proof}

Thus, we have an $\as^{\antishriek}$-coalgebra $V$. Notice that, since $d(V)\subseteq V^{\otimes 2}$, the underlying chain complex $V$ of the $\as^{\antishriek}$-coalgebra has the zero differential. We define
\[
\Phi:\as(V)\longrightarrow\as(V)
\]
by
\[
\Phi(v_n) = \sum_{k\ge1}\sum_{i_1+\cdots+i_k = n}v_{i_1}\otimes\cdots\otimes v_{i_k}\ .
\]

\begin{lemma}
	The map $\Phi$ commutes with the differential, and therefore defines an $\infty_\kappa$-morphism $\Phi:V\rightsquigarrow V$.
\end{lemma}

\begin{proof}
	We have to show that $\Phi:V\to\Omega_\kappa V$ satisfies the Maurer--Cartan equation
	\[
	\partial(\Phi)+\star_\kappa(\Phi) = 0\ .
	\]
	We have
	\begin{align*}
		\partial(\Phi)(v_n)&\ = d_{\Omega_\kappa V}\left(\sum_{k\ge1}\sum_{i_1+\cdots+i_k = n}v_{i_1}\otimes\cdots\otimes v_{i_k}\right)\\
		=&\ \sum_{\substack{k\ge1\\i_1+\cdots+i_k = n}}\sum_{j=0}^{k-1}(-1)^{i_1+\cdots+i_j}v_{i_1}\otimes\cdots\otimes v_{i_j}\otimes\\
		&\qquad\qquad\qquad\otimes\left(\sum_{\alpha+\beta = i_{j+1}-1}(-1)^\alpha v_\alpha\otimes v_\beta\right)\otimes v_{i_{j+2}}\otimes\cdots\otimes v_{i_k}\\
		=&\ \sum_{a,b\ge1}\sum_{\substack{x_1+\cdots+x_a+\\+y_1+\cdots+y_b = n-1}}(-1)^{x_1+\cdots+x_a}v_{x_1}\otimes\cdots\otimes v_{x_a}\otimes v_{y_1}\otimes\cdots\otimes v_{y_b}
	\end{align*}
	where in the first line we used the fact that $d_V = 0$, and in the last line we substituted $a=j+1,b=k-j+1$, $x_s = i_s$ for $s\le a$, $x_a = \alpha$, $y_1 = \beta$, and $y_s = i_{j+s}$ for $s\ge2$. At the same time, we have
	\begin{align*}
		\star_\kappa(\Phi)(v_n) =&\ (\gamma_\as\circ 1)(\kappa\circ\Phi)\Delta_V(v_n)\\
		=&\ -\sum_{i+j=n-1}(-1)^i\Phi(v_i)\otimes\Phi(v_j)\\
		=&\ -\sum_{a,b\ge1}\sum_{\substack{x_1+\cdots+x_a+\\+y_1+\cdots+y_b = n-1}}(-1)^{x_1+\cdots+x_a}v_{x_1}\otimes\cdots\otimes v_{x_a}\otimes v_{y_1}\otimes\cdots\otimes v_{y_b}\ .
	\end{align*}
	Notice the sign in the second line: it comes from the signs in the definition of the differential $d_2$ in the cobar construction. This concludes the proof.
\end{proof}

\subsection{The counterexample}

We now prove what claimed at the beginning of the present section by considering $\C=\as^{\antishriek}$, $\P=\as$, the canonical twisting morphism
\[
\kappa:\as^{\antishriek}\longrightarrow\as\ ,
\]
the associative algebras $A=A^2$, $A'=H^2$, the $\as^{\antishriek}$-coalgebras $D'=D=V$, and the $\infty$-morphisms $\Psi=i_\infty$ and $\Phi$ described above. We take the linear maps $f_1,f_2,f_3:V\to H^2$ such that $f_i(v_1) = z$ for $i=1,2,3$, $f_1(v_2) = z$, and $f_2(v_2) = f_3(v_2) = 0$. Notice that $f_2$ and $f_3$ have degree $-1$, while $f_1$ decomposes as the sum of a degree $-1$ map and a degree $-2$ map.

\medskip

We start by computing how $\hom^\kappa_\ell(\Phi,1)\hom^\kappa_r(1,i_\infty)(\mu^\vee_3\otimes F)$ acts on $v_4\in V$. We have
\begin{align*}
	\Phi(v_4&) =\\
	=&\ \id\otimes v_4 + \mu_2\otimes(v_1\otimes v_3 + v_2\otimes v_2 + v_3\otimes v_1)\\
	& + \mu_3\otimes(v_1\otimes v_1\otimes v_2 + v_1\otimes v_2\otimes v_1 + v_2\otimes v_1\otimes v_1) + \mu_4\otimes v_1\otimes v_1\otimes v_1\otimes v_1
\end{align*}
Since we will project on the part with only three copies of $V$, we don't care about the last term and will omit it in the following step. Notice that the comultiplication of $V$ is explicitly given by
\begin{align*}
\Delta_V(v_n) =&\ \id\otimes v_n + \sum_{i_1+i_2 = n-1}(-1)^{i_1}\susp_2^{-1}\mu_2^\vee\otimes v_{i_1}\otimes v_{i_2} -\\
&- \sum_{j_1+j_2+j_3 = n-2}(-1)^{j_2}\susp_3^{-1}\mu_3^\vee\otimes v_{j_1}\otimes v_{j_2}\otimes v_{j_3} + \cdots\ ,
\end{align*}
where the dots indicate terms with at least $4$ copies of $V$. Applying this to the above, and then using $\mathrm{proj}_3$, we get
\begin{align*}
	\mathrm{proj}_3&\as(\Delta_V)\Phi(v_4) =\\
	=&\ -\mu_2\otimes\Big((\id\otimes v_1)\otimes(\susp_2^{-1}\mu_2^\vee\otimes v_1\otimes v_1) + (\susp_2^{-1}\mu_2^\vee\otimes v_1\otimes v_1)\otimes(\id\otimes v_1)\Big)\\
	&\ + \mu_3\otimes\Big((\id\otimes v_1)\otimes(\id\otimes v_1)\otimes(\id\otimes v_2) + (\id\otimes v_1)\otimes(\id\otimes v_2)\otimes(\id\otimes v_1)\\
	&\qquad\qquad + (\id\otimes v_2)\otimes(\id\otimes v_1)\otimes(\id\otimes v_1)\Big)
\end{align*}
Applying $F$ gives
\begin{align*}
	F\mathrm{proj}_3\as(\Delta_V)&\Phi(v_4) =\\
	=&\ \mu_2\otimes\Big((\id\otimes z)\otimes(\susp_2^{-1}\mu_2^\vee\otimes z\otimes z) - (\susp_2^{-1}\mu_2^\vee\otimes z\otimes z)\otimes(\id\otimes z)\Big)\\
	&\ - \mu_3\otimes\big((\id\otimes z)\otimes(\id\otimes z)\otimes(\id\otimes z)\big)\ ,
\end{align*}
and thus
\begin{align*}
	\as(i_\infty)F\mathrm{proj}_3\as(\Delta_V)&\Phi(v_4) = \mu_2\otimes(x\otimes y - y\otimes x) - \mu_3\otimes x\otimes x\otimes x\ .
\end{align*}
Finally, we have
\begin{align*}
	\hom^\kappa_\ell(\Phi,1)\hom^\kappa_r(1,i_\infty)(\mu^\vee_3\otimes F)(v_4) =&\ \gamma_{A_2}\as(i_\infty)F\mathrm{proj}_3\as(\Delta_V) = -x^3.
\end{align*}
Now we look at the action of $\hom^\kappa_r(1,i_\infty)\hom^\kappa_\ell(\Phi,1)(\mu_3^\vee\otimes F)$ on $v_4$. We have
\[
\Delta_V(v_4) = \id\otimes v_4 + \susp_2^{-1}\mu_2^\vee\otimes(-v_1\otimes v_2 + v_2\otimes v_1)\ ,
\]
and thus
\begin{align*}
	\mathrm{proj}_3&\as^{\antishriek}(\Phi)\Delta_V(v_4) =\\
	=&\ \id\otimes\mu_3\otimes(v_1\otimes v_1\otimes v_2 + v_1\otimes v_2\otimes v_1 + v_2\otimes v_1\otimes v_1)\\
	&\ + \susp_2^{-1}\mu_2^\vee\otimes\big(-(\id\otimes v_1)\otimes(\mu_2\otimes v_1\otimes v_1) + (\mu_2\otimes v_1\otimes v_1)\otimes(\id\otimes v_1)\big)\ .
\end{align*}
Applying $F$ we obtain
\begin{align*}
	F\mathrm{proj}_3&\as^{\antishriek}(\Phi)\Delta_V(v_4) =\\
	=&\ -\id\otimes\mu_3\otimes z\otimes z\otimes z +\\
	&+ \susp_2^{-1}\mu_2^\vee\otimes\big(-(\id\otimes z)\otimes(\mu_2\otimes z\otimes z) + (\mu_2\otimes z\otimes z)\otimes(\id\otimes z)\big)\ ,
\end{align*}
and thus
\begin{align*}
	\as^{\antishriek}(\gamma_{H^2})F\mathrm{proj}_3\as^{\antishriek}(\Phi)&\Delta_V(v_4) = 0
\end{align*}
since $z^2 = 0$ in $H^2$ and by \cref{lemma:An is formal}. Therefore,
\[
\hom^\kappa_r(1,i_\infty)\hom^\kappa_\ell(\Phi,1)(\mu_3^\vee\otimes F)(v_4) = 0\ ,
\]
showing that
\[
\hom^\kappa_r(1,i_\infty)\hom^\kappa_\ell(\Phi,1)\neq\hom^\kappa_\ell(\Phi,1)\hom^\kappa_r(1,i_\infty)
\]
as claimed. This implies the result we wanted.

\begin{theorem}
	In general, there is no bifunctor
	\[
	\hom^\alpha:(\mathsf{conil.}\ \infty_\alpha\text{-}\Ccog)^{op}\times\infty_\alpha\text{-}\Palg\longrightarrow\infty\text{-}\Lalg.
	\]
	that restricts to the functors
	\[
	\hom^\alpha_\ell:(\mathsf{conil.}\ \infty_{\alpha}\text{-}\Ccog)^{op}\times\Palg\longrightarrow\infty\text{-}\Lalg\ 
	\]
	and
	\[
	\hom^\alpha_r:\mathsf{conil.}\ \Ccog^{op}\times\infty_\alpha\text{-}\Palg\longrightarrow\infty\text{-}\Lalg
	\]
	defined above in the respective subcategories.
\end{theorem}

\begin{remark}
	The result is true in any characteristic in the non-symmetric case by the same counterexample as above, and in the symmetric case as well, by considering the same counterexample and tensoring the operads by the regular representation of the symmetric groups.
\end{remark}

\section{Analogous results in different settings} \label{sect:differentSettings}

All the results presented above have analogous incarnation in different contexts. We present here the cases of non-symmetric operads, where everything works over fields of any characteristic, and the dual case, where we take tensor products of algebras instead of convolution algebras.

\subsection{Tensor products}

Let $\C$ be a cooperad, and let $\P$ be an operad. It is a well known fact that the dual of a cooperad is always an operad, and if we further assume that $\C$ is finite dimensional in every arity, then we have a canonical isomorphism of $\S$-modules
\[
\hom(\C,\P)\cong\P\otimes\C^\vee,
\]
where $\otimes$ is the Hadamard tensor product (i.e. the arity-wise tensor product of $\S$-modules). The operad structure making the left-hand side into the convolution operad naturally induces an operad structure on the right hand side.

\medskip

We have the following result, analogous to \cref{thm:bijTwHoms}.

\begin{theorem} \label{thm:bijTwTensor}
	Suppose $\C$ is finite dimensional in every arity. There is a bijection
	\[
	\Tw(\C,\P)\stackrel{\cong}{\longrightarrow}\hom_{\op}(\L_\infty,\P\otimes\C^\vee)\ .
	\]
	This is natural in a sense analogous to what explained in \cref{thm:bijTwHoms}.
\end{theorem}

\begin{remark}
	A hint to this idea can already be found in Ginzburg--Kapranov \cite[Prop. 3.2.18]{gk94} and \cite[Appendix C]{bl15}. A special case, namely the one already mentioned in \cref{rem:SHLpaperRN}, was studied in \cite{rn17}.
\end{remark}

As a corollary, we obtain a bifunctor
\[
\otimes^\alpha:\P\text{-}\mathsf{alg}\times\C^\vee\text{-}\mathsf{alg}\longrightarrow\Lalg
\]
taking a $\P$-algebra $A$ and a $\C^\vee$-algebra $D$ and giving back the chain complex $A\otimes D$ with the $\L_\infty$-algebra structure induced by \cref{thm:bijTwTensor} above.

\begin{remark}
	One can give a result analogous to \cref{thm:MCelOfConvAlg} also in this setting. However, in this case one must assume that the $\C^\vee$-coalgebras under consideration are finite dimensional, plus some completeness assumptions in order for the Maurer--Cartan equation to be well-defined. See \cite[Cor. 6.6]{rn17} for a special case of this.
\end{remark}

An explicit formula for the $\L_\infty$-algebra structure is the following one. Fix a basis $\{c_i\}_i$ of $\C(n)$, and let $\{c_i^\vee\}_i\in\C(n)^\vee$ be the dual basis. Denote by $p_i\coloneqq\alpha(c_i)$. Then
\[
\ell_n(a_1\otimes x_1,\ldots,a_n\otimes x_n) = \sum_i(-1)^\epsilon\gamma_A(p_i\otimes a_1\otimes\cdots\otimes a_n)\otimes\gamma_C(c_i^\vee\otimes x_1\otimes\ldots\otimes x_n)\ ,
\]
for $a_j\in A$ and $x_j\in C$, where
\[
\epsilon = \sum_{k=1}^n|a_k|\sum_{j=1}^{k-1}|x_j| + |c_i|\sum_{k=1}^n|a_k|
\]
is the obvious Koszul sign. Notice that this is independent from the choice of the basis.

\medskip

The question is now to understand how this behaves with respect to the appropriate notions of $\infty$-morphism in both slots. A first question one has to ask is: what is the good notion of homotopy morphism for $\C^\vee$-algebras?

\begin{lemma}
	Let $\C$ be a cooperad which is finite-dimensional in every arity, and let $\P$ be an operad which is finite-dimensional in every arity. Then $\alpha:\C\to\P$ is a twisting morphism if and only if $\alpha^\vee:\P^\vee\to\C^\vee$ is a twisting morphism.
\end{lemma}

\begin{proof}
	Since we work in the finite dimensional case, it is sufficient to prove one direction. A twisting morphism $\alpha:\C\to\P$ is the same thing as a collection of elements
	\[
	\alpha(n)\in\left(\P(n)\otimes\C(n)^\vee\right)^{\S_n},
	\]
	but
	\[
	\left(\P(n)\otimes\C(n)^\vee\right)^{\S_n}\cong\left(\C(n)^\vee\otimes\P(n)^{\vee\vee}\right)^{\S_n}
	\]
	and one recovers $\alpha^\vee$ from $\alpha$ through this isomorphism.
\end{proof}

Now let $A$ be a $\P$-algebra, let $C,C'$ be two $\C^\vee$-algebras and let
\[
g:\bar_{\alpha^\vee}C\longrightarrow\bar_{\alpha^\vee}C'
\]
be an $\infty_{\alpha^\vee}$-morphism $g:C\rightsquigarrow C'$ of $\C^\vee$-algebras. We construct a morphism (not respecting the differentials \emph{a priori})
\[
1\otimes^\alpha g:\bar_\iota(A\otimes^\alpha C)\longrightarrow\bar_\iota(A\otimes^\alpha C')
\]
as follows. For each $n$, fix a basis $\{p_i\}_i$ of $\P(n)$ (where we leave the $n$ out of the notation). Then the identity of $\P(n)$ is the element $\sum_ip_i\otimes p_i^\vee\in\P(n)\otimes\P(n)^\vee$, where $\{p_i^\vee\}_i\in\P(n)^\vee$ is the dual basis. The morphism $1\otimes^\alpha g$ is the unique morphism of cofree cocommutative coalgebras extending the map sending
\[
\mu_n^\vee\otimes(a_1\otimes x_1)\otimes\cdots\otimes(a_n\otimes x_n)\in\com(n)^\vee\otimes_{\S_n}(A\otimes C)^{\otimes n}
\]
to
\[
\sum_i(-1)^\epsilon\gamma_A(p_i\otimes a_1\otimes\cdots\otimes a_n)\otimes g_n(p_i^\vee\otimes x_1\otimes\cdots\otimes x_n)\ ,
\]
where
\[
\epsilon = \sum_{k=1}^n|a_k|\sum_{j=1}^{k-1}|x_j| + |p_i|\sum_{k=1}^n|a_k|\ .
\]
Dually, given two $\P$-algebras $A,A'$, a $\C^\vee$-algebra $C$ and an $\infty_\alpha$-morphism $f:A\rightsquigarrow A'$, we have a morphism
\[
f\otimes^\alpha1:\bar_\iota(A\otimes^\alpha C)\longrightarrow\bar_\iota(A'\otimes^\alpha C)
\]
by sending
\[
\mu_n^\vee\otimes(a_1\otimes x_1)\otimes\cdots\otimes(a_n\otimes x_n)\in\com(n)^\vee\otimes_{\S_n}(A\otimes C)^{\otimes n}
\]
to
\[
\sum_i(-1)^\epsilon f_n(c_i\otimes a_1\otimes\cdots\otimes a_n)\otimes \gamma_C(c_i^\vee\otimes x_1\otimes\cdots\otimes x_n)\ ,
\]
where $\{c_i\}_i\in\C(n)$ is a basis, $\{c_i^\vee\}_i\in\C(n)^\vee$ is the dual basis, and
\[
\epsilon = \sum_{k=1}^n|a_k|\sum_{j=1}^{k-1}|x_j| + |c_i|\sum_{k=1}^n|a_k|\ .
\]
The analogue of \cref{thm:twoBifunctors} in this context is the following.

\begin{theorem} \label{thm:twoInftyMorphismsTensor}
	The morphisms $f\otimes^\alpha 1$ and $1\otimes^\alpha g$ defined above are $\infty$-morphisms of $\L_\infty$-algebras.
\end{theorem}

\begin{proof}
	The proof is essentially dual to the proof of \cref{thm:twoBifunctors}. One has to choose bases and work with them. An essential step is to characterize the properties of the elements $\alpha(n)\in\P(n)\otimes\C(n)^\vee$ corresponding to the fact that $\alpha$ satisfies the Maurer--Cartan equation.
\end{proof}

\begin{remark}
	This result is a generalization of \cite[Prop. 4.4]{rn17} by the first author. The case presented there is the following. We are given two binary quadratic operads with a finite dimensional generating space for $\Q$ and a morphism
	\[
	\Psi:\Q\longrightarrow\P
	\]
	between them, a $\P$-algebra $A$ and two $\Omega\Q^\vee$-algebras $C,C'$ (i.e. $\Q^!_\infty$-algebras, up to a suspension). Then we have $\L_\infty$-algebras $A\otimes^\Psi C$ and $A\otimes^\Psi C'$ (see \cref{rem:SHLpaperRN}). If we have an $\infty$-morphism
	\[
	g:\bar_\iota C\longrightarrow\bar_\iota C'
	\]
	of $\Q^!_\infty$-algebras, we obtain an $\infty$-morphism of $\L_\infty$-algebras
	\[
	1\otimes^\Psi g:\bar_\iota(A\otimes^\Psi C)\longrightarrow\bar_\iota(A\otimes^\Psi C')
	\]
	by sending $\mu_n^\vee\otimes(a_1\otimes c_1)\otimes\cdots(a_n\otimes c_n)$ to
	\[
	\sum_i(-1)^\epsilon\gamma_A(\Psi(q_i)\otimes a_1\otimes\cdots\otimes a_n)\otimes g_n(q_i^\vee\otimes c_1\otimes\cdots\otimes c_n)\ ,
	\]
	with $\epsilon$ the appropriate Koszul sign. The difference here is that the $\infty$-morphism $g$ is relative to the twisting morphism $\iota:\Q^\vee\to\Omega\Q^\vee$, and not with respect to the dual of the twisting morphism inducing the $\L_\infty$-algebra structure on the tensor products, that is
	\[
	\psi=\left(\Bar\Q\stackrel{\pi}{\longrightarrow}\Q\stackrel{\Psi}{\longrightarrow}\P\right)\ .
	\]
	We recover this from our \cref{thm:twoInftyMorphismsTensor} by using $\Psi:\Q\to\P$ to see $A$ as a $\Q$-algebra $\Psi^*A$. Then
	\[
	A\otimes^\Psi C = A\otimes^\psi C = (\Psi^*A)\otimes^\pi C
	\]
	and
	\[
	1\otimes^\Psi g = 1\otimes^\pi g\ .
	\]
\end{remark}

\begin{theorem}
	The bifunctor
	\[
	\otimes^\alpha:\P\text{-}\mathsf{alg}\times\C^\vee\text{-}\mathsf{alg}\longrightarrow\Lalg
	\]
	extends to bifunctors
	\[
	\otimes^\alpha_\ell:\infty_\alpha\text{-}\Palg\times\C^\vee\text{-}\mathsf{alg}\longrightarrow\Lalg\ .
	\]
	and
	\[
	\otimes^\alpha_r:\Palg\times\infty_{\alpha^\vee}\text{-}\C^\vee\text{-}\mathsf{alg}\longrightarrow\Lalg
	\]
\end{theorem}

By dualizing the coalgebras, the example treated in \cref{sect:counterexample} proves that one can not extend $\otimes^\alpha$ to a bifunctor taking $\infty$-morphisms in both slots.

\begin{theorem}
	In general, there is no bifunctor
	\[
	\otimes^\alpha:\infty_\alpha\text{-}\Palg\times\infty_{\alpha^\vee}\text{-}\C^\vee\text{-}\mathsf{alg}\longrightarrow\infty\text{-}\Lalg.
	\]
	that restricts to the functors $\otimes^\alpha_\ell$ and $\otimes^\alpha_r$ defined above in the respective subcategories.
\end{theorem}

\begin{proof}
	The counterexample in this setting is obtained by dualizing the coalgebras in the counterexample given in \cref{sect:counterexample}. Notice that one can have $f_1,f_2,f_3$ of finite rank by setting the image of all $v_i$ to zero for $i\ge3$, and such maps can be represented as elements of $A^2\otimes V^\vee$.
\end{proof}

\subsection{Non-symmetric operads}\label{subsect:nsCase}

If we consider the same situation as in the rest of the article, but working with non-symmetric operads, then all the results still hold after replacing the operad $\L_\infty$ with the operad $\A_\infty$ encoding suspended homotopy associative algebras. The Maurer--Cartan elements one considers in a $\A_\infty$-algebra $A$ are the elements $x\in A_0$ such that
\[
dx + \sum_{n\ge2}m_n(x,\ldots,x) = 0\ ,
\]
where $m_n$ is the generator of arity $n$ of $\A_\infty$. Moreover, since in this setting we don't need to identify invariants and coinvariants, we are allowed to use any base field, without restrictions on the characteristic.

\section{Invariance of Maurer--Cartan spaces}\label{sect:MC spaces}

Given an $\L_\infty$-algebra, one can associate to it a Kan complex --- the Deligne--Hinich--Getzler $\infty$-grou\-poid, or deformation $\infty$-grou\-poid --- \cite{Hin97}, \cite{Get09}. In this section, we study the homotopical behavior of the deformation $\infty$-groupoids associated to convolution homotopy Lie algebras under $\alpha$-weak equivalences and $\infty_\alpha$-quasi-isomorphisms.

\subsection{Deformation $\infty$-groupoids and the Dolgushev--Rogers theorem}

We start by recalling the deformation $\infty$-groupoid of an $\L_\infty$-algebra. This object, first defined in \cite{Hin97} and \cite{Get09}, plays a fundamental role in modern deformation theory and rational homotopy theory. It encodes Maurer--Cartan elements, equivalences between them, and higher compatibilities, all at once in a simplicial set.

\medskip

In order to define this object, we first need to introduce the notion of a filtered $\L_\infty$-algebra.

\begin{definition} \label{def:filteredLoo}
	A \emph{filtered $\L_\infty$-algebra} is an $\L_\infty$-algebra $\g$ together with a descending filtration
	\[
	\g=\F_1\g\supseteq\F_2\g\supseteq\F_3\g\supseteq\cdots
	\]
	satisfying:
	\begin{itemize}
		\item[i)] $d(\F_n\g)\subseteq\F_n\g$,
		\item[ii)] $\ell_k(\F_{n_1}\g,\ldots,\F_{n_k}\g)\subseteq		     \F_{n_1+\cdots+n_k}\g$, and
		\item[iii)] $\g\cong\lim_n\g/\F_n\g$ as $\L_\infty$-algebras.
	\end{itemize}
	Let $(\g,\F\g)$ and $(\h,\F\h)$ be two filtered $\L_\infty$-algebras. An \emph{$\infty$-morphism of filtered $\L_\infty$-algebras}
	\[
	\Phi:(\g,\F\g)\rightsquigarrow(\h,\F\h)
	\]
	is an $\infty$-morphism of $\L_\infty$-algebras $\Phi:\g\rightsquigarrow\h$ such that
	\[
	\phi_k(\F_{n_1}\g,\ldots,\F_{n_k}\g)\subseteq\F_{n_1+\cdots+n_k}\h\ .
	\]
	It is called a \emph{filtered $\infty$-quasi-isomorphism} if moreover we have that all the induced morphisms
	\[
	\phi_1|_{\F_n\g}:\F_n\g\longrightarrow\F_n\h
	\]
	are quasi-isomorphisms. The category of filtered $\L_\infty$-algebras and filtered $\infty$-mor\-phisms between them will be denoted by $\FLalg$.
\end{definition}

\begin{example}\label{ex:complete Loo algebras}
	Let $\g$ be an $\L_\infty$-algebra. Its \emph{canonical filtration} --- also known as the \emph{lower central series} --- is the descending filtration defined as
	\[
	\F_n^{can}\g\coloneqq\{x\in\g\mid x\text{ can be written as a bracketing of at least }n\text{ elements}\}\ .
	\]
	for $n\ge1$. This filtration always satisfies conditions (i) and (ii) of \cref{def:filteredLoo}. When it also satisfied condition (iii), so that $(\g,\F_\bullet^{can}\g)$ is a filtered $\L_\infty$-algebra, we simply say that $\g$ is \emph{complete}. Notice that every nilpotent $\L_\infty$-algebra is complete, since its canonical filtration will terminate at some point.
\end{example}

We denote by $\Omega_\bullet$ the Sullivan algebra of polynomial de Rham forms on the simplex. It is the simplicial commutative algebra given by
\[
\Omega_n\coloneqq\frac{\k[t_0,\ldots,t_n,dt_0,\ldots,dt_n]}{\left(\sum_{i=0}^nt_i = 1,\sum_{i=0}^ndt_i = 0\right)}\ ,
\]
with $|t_i|=0$ and $|dt_i|=-1$ for all $0\le i\le n$ and differential induced by $d(t_i) = dt_i$. It can be interpreted as the algebra of polynomial differential forms on the standard geometric $n$-simplex $\Delta^n$, and it was introduced by Sullivan in \cite{Sul77} as a commutative rational model for the $n$-simplex. Notice that we reversed the degrees, as we are working over chain complexes.

\begin{definition}
	Let $(\g,\F\g)$ be a filtered $\L_\infty$-algebra. The \emph{Deligne--Hinich--Getz\-ler $\infty$-groupoid}, also called the \emph{deformation $\infty$-groupoid}, of $(\g,\F\g)$ is the simplicial set whose $n$-simplices are given by
	\[
	\MC_n(\g,\F\g)\coloneqq\lim_r\MC(\g/\F_r\g\otimes\Omega_n)\ ,
	\]
	and whose face and degeneracy maps are induced by the face and degeneracy maps of $\Omega_{\bullet}$. We will sometimes omit the filtration from the notation when this is clear from the context.
\end{definition}

As a matter of fact, the assignment $\MC_\bullet$ can be extended to a functor on the category of filtered $\L_\infty$-algebras and filtered $\infty$-morphisms between them
\[
\MC_\bullet:\FLalg\longrightarrow\ssets\ ,
\]
see e.g. \cite[Sect. 2.2]{dr15}.

\medskip

A very important result, due to Hinich and Getzler, is the fact that the functor $\MC_\bullet$ takes values in Kan complexes, which motivates calling it an $\infty$-groupoid.

\begin{theorem}[{\cite[Prop. 2.2.3]{Hin97},\cite[Prop. 4.7]{Get09}}]
	Let $(\g,\F\g)$ be a filtered $\L_\infty$-algebra. The simplicial set $\MC_\bullet(\g,\F\g)$ is a Kan complex.
\end{theorem}

Another fundamental result in this context, which we will need in what follows, is the Dolgushev--Rogers theorem \cite[Thm. 2.2]{dr15}.

\begin{theorem}[Dolgushev--Rogers]\label{thm:DR}
	Let $(\g,\F\g)$ and $(\h,\F\h)$ be two filtered $\L_\infty$-algebras, and let
	\[
	\Phi:(\g,\F\g)\rightsquigarrow(\h,\F\h)
	\]
	be a filtered $\infty$-quasi-isomorphism. Then the morphism of simplicial sets
	\[
	\MC_\bullet(\Phi):\MC_\bullet(\g,\F\g)\longrightarrow\MC_\bullet(\h,\F\h)
	\]
	is a homotopy equivalence.
\end{theorem}

\subsection{Filtered algebras and $\infty$-quasi-isomorphisms}

The notion of filtered algebras is not exclusive to $\L_\infty$-algebras. We define here the analogous notion for algebras over an arbitrary operad $\P$ and go on to state the main result of this section, giving us invariance of the Maurer--Cartan spaces of convolution algebras under (filtered) $\infty_\alpha$-quasi-isomorphisms.

\medskip

For the rest of the section, fix a cooperad $\C$, an operad $\P$, and a Koszul twisting morphism $\alpha:\C\to\P$.

\begin{definition}
	A \emph{filtered $\P$-algebra} is a $\P$-algebra $A$ together with a descending filtration
	\[
	A=\F_1A\supseteq\F_2A\supseteq\F_3A\supseteq\cdots
	\]
	satisfying:
	\begin{itemize}
		\item[i)] $d(\F_nA)\subseteq\F_nA$,
		\item[ii)] $\gamma_A(p\otimes\F_{n_1}A\otimes\cdots\otimes\F_{n_k}A)\subseteq\F_{n_1+\cdots+n_k}A$ for any $p\in\P(k)$, and
		\item[iii)] The canonical map $A\to\lim_nA/\F_nA$ is an isomorphism of $\P$-algebras.
	\end{itemize}
	Let $(A,\F A)$ and $(A',\F A')$ be two filtered $\P$-algebras. An \emph{$\infty_\alpha$-morphism of filtered $\P$-algebras}
	\[
	\Phi:(A,\F A)\rightsquigarrow(A',\F A')
	\]
	is an $\infty_\alpha$-morphism of $\P$-algebras $\Phi:A\rightsquigarrow A'$ such that
	\[
	\phi_k(c\otimes\F_{n_1}A\otimes\cdots\otimes\F_{n_k}A)\subseteq\F_{n_1+\cdots+n_k}A'
	\]
	for any $c\in\C(k)$. It is called a \emph{filtered $\infty_\alpha$-quasi-isomorphism} if moreover we have that all the induced morphisms
	\[
	\phi_1|_{\F_nA}:\F_nA\longrightarrow\F_nA'
	\]
	are quasi-isomorphisms.
\end{definition}

\begin{remark}
	There is a dual notion of cofiltered $\C$-coalgebra. We will not need it in the present paper, and thus we will avoid giving the details and the results dual to those we will state for algebras, but everything works similarly.
\end{remark}

Given a conilpotent $\C$-coalgebra and a filtered $\P$-algebra, we can put a natural filtration on their convolution $\L_\infty$-algebra.

\begin{proposition}
	Let $D$ be a conilpotent $\C$-coalgebra, and let $(A,\F A)$ be a filtered $\P$-algebra. Then the filtration
	\[
	\F_n\hom^\alpha(D,A)\coloneqq\hom^\alpha(D,\F_nA)
	\]
	makes $\hom^\alpha(D,A)$ into a filtered $\L_\infty$-algebra.
\end{proposition}

\begin{proof}
	The filtration is obviously descending and $\F_1\hom^\alpha(D,A) = \hom^\alpha(D,A)$. We have to show that the three conditions of \cref{def:filteredLoo} are satisfied.
	
	\medskip
	
	For the first condition, let $f\in\F_n\hom(D,A)$ and let $x\in D$. Then
	\[
	\partial(f)(x) = d_Af(x) - (-1)^ff(d_Dx)\in\F_nA\ ,
	\]
	since $d_A(\F_nA)\subseteq\F_nA$ by definition. Therefore, we have that
	\[
	\partial(f)\in\F_n\hom(D,A)\ .
	\]
	For the second condition, let $k\ge2$, fix $n_1,\ldots,n_k\ge1$ and let $f_i\in\F_{n_i}\hom^\alpha(D,A)$ for $1\le i\le k$. Take $x\in D$ and denote $F\coloneqq f_1\otimes\cdots\otimes f_k$ as usual. We have
	\[
	\ell_k(f_1,\ldots,f_k)(x) = \gamma_A(\alpha\circ1_D)F\Delta^k_D(x)\in\F_{n_1+\cdot+n_k}\hom^\alpha(D,A)
	\]
	since
	\[
	F\Delta^k_D(x)\in\C(k)\otimes_{\S_k}(\F_{n_1}A\otimes\cdots\otimes\F_{n_k}A)
	\]
	and by the fact that $\F A$ makes $A$ into a filtered $\P$-algebra. Thus, the second condition also holds.
	
	\medskip
	
	For the third condition, we have an isomorphism of chain complexes
	\[
	\frac{\hom^\alpha(D,A)}{\F_n\hom^\alpha(D,A)}\stackrel{\cong}{\longrightarrow}\hom^\alpha(D,A/\F_nA)
	\]
	given by sending the class of a linear map $f:D\to A$ to the composite
	\[
	D\stackrel{f}{\longrightarrow}A\xrightarrow{\mathrm{proj.}}A/\F_nA\ .
	\]
	It is straightforward to check that this is in fact an isomorphism of $\L_\infty$-alge\-bras. Therefore, we have
	\begin{align*}
		\lim_n\frac{\hom^\alpha(D,A)}{\F_n\hom^\alpha(D,A)}\cong&\ \lim_n\hom^\alpha(D,A/\F_nA)\\
		\cong&\ \hom^\alpha(D,\lim_nA/\F_nA)\\
		\cong&\ \hom^\alpha(D,A)\ .
	\end{align*}
	The isomorphism of the second line holds \emph{a priori} for chain complexes, but once again it is straightforward to check that it is also true at the level of $\L_\infty$-algebras. This completes the proof.
\end{proof}

This filtration is well behaved with respect to $\infty$-morphisms of $\L_\infty$-algebras induced through \cref{thm:twoBifunctors}. More precisely, we have the following.

\begin{proposition}\label{prop:induced morphisms are filtered}
	Let $D,D'$ be two $\C$-coalgebras, and let $(A,\F A),(A',\F A')$ be two filtered $\P$-algebras. We endow the various convolution $\L_\infty$-algebras with the respective filtrations defined as above.
	\begin{enumerate}
		\item\label{part1} Let $\Phi:D'\rightsquigarrow D$ be an $\infty_\alpha$-morphism of conilpotent $\C$-coalgebras. The induced $\infty$-morphism between the convolution $\L_\infty$-algebras
		\[
		\hom_\ell^\alpha(\Phi,1):\hom^\alpha(D,A)\rightsquigarrow\hom^\alpha(D',A)
		\]
		is filtered.
		\item\label{part2} Let $\Psi:(A,\F A)\rightsquigarrow(A',\F A')$ be a filtered $\infty_\alpha$-morphism of filtered $\P$-algebras. The induced $\infty$-morphism between the convolution $\L_\infty$-algebras
		\[
		\hom_r^\alpha(1,\Psi):\hom^\alpha(D,A)\rightsquigarrow\hom^\alpha(D,A')
		\]
		is filtered.
	\end{enumerate}
\end{proposition}

\begin{proof}
	We start by proving (\ref{part1}). Take elements $f_1,\ldots,f_k\in\hom^\alpha(D,A)$ with $f_i\in\F_{n_i}\hom^\alpha(D,A)$, and let $x\in D$. Then
	\[
	\hom_\ell^\alpha(\Phi,1)(f_1,\ldots,f_k)(x) = \gamma_AF\phi_k(x)\in\F_{n_1+\cdots+n_k}A
	\]
	since
	\[
	F\phi_k(x)\in\C(k)\otimes_{\S_k}(\F_{n_1}A\otimes\cdots\otimes\F_{n_k}A)
	\]
	and the third condition defining filtered $\P$-algebras. Thus,
	\[
	\hom_\ell^\alpha(\Phi,1)(f_1,\ldots,f_k)\in\F_{n_1+\cdots+n_k}\hom^\alpha(D,A)
	\]
	showing that $\hom_\ell^\alpha(\Phi,1)$ is a filtered $\infty$-morphism.
	
	\medskip
	
	Part (\ref{part2}) is similar. Again, let $f_1,\ldots,f_k\in\hom^\alpha(D,A)$ with $f_i\in\F_{n_i}\hom^\alpha(D,A)$, and let $x\in C$. Then
	\[
	\hom_r^\alpha(\Phi,1)(f_1,\ldots,f_k)(x) = \phi_k F\Delta^k_D(x)\in\F_{n_1+\cdots+n_k}A
	\]
	since again
	\[
	F\Delta^C_k(x)\in\C(n)\otimes_{\S_k}(\F_{n_1}A\otimes\cdots\otimes\F_{n_k}A)
	\]
	and because of the fact that $\Psi$ is a filtered morphism. This shows that also $\hom_r^\alpha(1,\Psi)$ is a filtered $\infty$-morphism, concluding the proof.
\end{proof}

Finally, we state the main result of this section.

\begin{theorem}\label{thm:invariance of MC spaces}
	Let $D,D'$ be two conilpotent $\C$-coalgebras, and let $(A,\F A),(A',\F A')$ be two filtered $\P$-algebras. We endow the various convolution $\L_\infty$-algebras with the respective filtrations defined as above.
	\begin{enumerate}
		\item Let $\Phi:D'\rightsquigarrow D$ be an $\infty_\alpha$-quasi-isomorphism of $\C$-coalgebras. The induced morphism
		\[
		\MC_\bullet(\hom_\ell^\alpha(\Phi,1)):\MC_\bullet(\hom^\alpha(D,A))\longrightarrow\MC_\bullet(\hom^\alpha(D',A))
		\]
		is a weak equivalence of simplicial sets. In particular, by \cref{thm: alpha-we are oo-alpha-qi} this is true for $\alpha$-weak equivalences of conilpotent $\C$-coalgebras.
		\item Let $\Psi:(A,\F A)\rightsquigarrow(A',\F A')$ be a filtered $\infty_\alpha$-quasi-isomorphism of filtered $\P$-algebras. The induced morphism
		\[
		\MC_\bullet(\hom_r^\alpha(1,\Psi)):\MC_\bullet(\hom^\alpha(D,A))\longrightarrow\MC_\bullet(\hom^\alpha(D,A'))
		\]
		is a weak equivalence of simplicial sets.
	\end{enumerate}
\end{theorem}

\begin{proof}
	By \cref{prop:induced morphisms are filtered}, we already know that the induced $\infty$-morphisms of $\L_\infty$-algebras are filtered. We will show that the induced morphisms are filtered $\infty$-quasi-isomorphisms of $\L_\infty$-algebras, and then conclude by \cref{thm:DR}.
	
	\medskip
	
	For the $\infty$-morphism $\hom_\ell^\alpha(\Phi,1)$, since we assumed that $\Phi$ is an $\infty_\alpha$-quasi-iso\-mor\-phism, we have that $\hom_\ell^\alpha(\Phi,1)$ and all of the $\infty$-morphisms
	\[
	\hom_\ell^\alpha(\Phi,1_{A/\F_nA}):\hom^\alpha(D,A/\F_nA)\rightsquigarrow\hom^\alpha(D',A/\F_nA)
	\]
	are $\infty$-quasi-isomorphisms by \cref{prop:qiInduceQi}, i.e. that their component in arity $1$ is a quasi-iso\-mor\-phism. We have
	\[
	\hom^\alpha(D,A/\F_nA)\cong\frac{\hom^\alpha(D,A)}{\F_n\hom^\alpha(D,A)}\ ,
	\]
	giving the short exact sequence
	\[
	0\longrightarrow\F_n\hom^\alpha(D,A)\longrightarrow\hom^\alpha(D,A)\longrightarrow\hom^\alpha(D,A/\F_nA)\longrightarrow0\ .
	\]
	The associated long exact sequence in homology, together with the $5$-lemma, implies that $\hom_\ell^\alpha(\Phi,1)$ is a filtered $\infty$-quasi-isomorphism.
	
	\medskip
	
	For the $\infty$-morphism $\hom_r^\alpha(1,\Psi)$, we assumed that $\Psi$ is a filtered $\infty_\alpha$-quasi-isomorphism. Therefore, all of the induced maps
	\[
	\psi^{(n)}_1:\g/\F_n\g\longrightarrow\g'/\F\g'
	\]
	are quasi-isomorphisms, again by an argument similar to the one presented for the first case. It follows that all of the $\infty$-morphisms
	\[
	\hom_r^\alpha(1,\Psi^{(n)}):\hom^\alpha(C,\g/\F_n\g)\rightsquigarrow\hom^\alpha(C,\g'/\F_n\g')
	\]
	are $\infty$-quasi-isomorphisms, and thus that $\hom^\alpha(1,\Psi)$ is a filtered $\infty$-quasi-iso\-mor\-phism.
\end{proof}

\section{Application: Rational models for mapping spaces} \label{sect:application}

In this final section, we present an application of the theory developed in the rest of the present paper to rational homotopy theory, using it to construct rational models for mapping spaces. We assume that the reader is familiar with the basics of rational homotopy theory, and redirect the reader to \cite{FHT01} for an accessible introduction.

\medskip

In the rest of this section, by spaces we mean simplicial sets. We only work with simply-connected spaces of finite $\q$-type. Moreover, we will often impose additional conditions on the spaces, such as the existence of a locally finite, degree-wise nilpotent $\L_\infty$-models for them (see \cref{def:locally finite}, \cref{def:degree-wise nilpotent}, and \cref{def:Loo models}). This last condition is not too stringent, as shown by \cref{prop:condition ok}.

\medskip

Let $K$ and $L$ be two such spaces, and suppose moreover that $K$ and  $L$ are based. Let $C$ be a homotopy cocommutative coalgebra model for $K$, and let $\g$ be an $\L_\infty$-algebra model for $L$. In \cite{w16}, the second author equipped the chain complex $\hom(C,\g)$ with an $\L_\infty$-algebra structure using \cref{thm:bijTwHoms}. We will prove that $\hom(C,\g)$ with this $\L_\infty$-algebra structure gives a rational model for the based mapping space $\map(K,L)$ (the basepoint being given by the constant map to the basepoint of $L$).

\medskip

Throughout this section, we will work over the field $\k=\q$  of rational numbers. By \emph{homotopy cocommutative coalgebras} we mean coalgebras over the cooperad $\bar\Omega\com^\vee$, which provides us with a resolution of $(\susp\otimes\lie)^{\antishriek}=\com^\vee$. We have the commutative diagram
\begin{center}
	\begin{tikzpicture}
		\node (a) at (0,1.5){$\com^\vee$};
		\node (b) at (0,0){$\bar\Omega\com^\vee$};
		\node (d) at (3,0){$\L_\infty$};
		
		\draw[->] (a) -- node[left]{$f_\iota$} (b);
		\draw[->] (a) edge[out=0,in=120] node[above]{$\iota$} (d);
		\draw[->] (b) -- node[above]{$\pi$} (d);
	\end{tikzpicture}
\end{center}
where the quasi-isomorphism $f_\iota$ is obtained by \cite[Prop. 6.5.8]{lodayvallette} --- it is in fact the unit of the bar-cobar adjunction --- and where all twisting morphisms are Koszul. We endow the three operads with the connected weight grading given by putting in weight $\omega$ the elements of arity $\omega+1$. The morphism $f_\iota$ and both twisting morphisms preserve the weight, so that we can use the results of Section \ref{subsect:oo-alpha-qis and rectifications}. Further, we suppose that all our algebras and coalgebras are simply-connected, i.e. concentrated in degrees greater or equal than $2$. This can always be done when working with simply-connected spaces, see \cite[Thm. 1]{Quil69}.

\subsection{Conditions on $\L_\infty$-algebras}

Later, we will need a condition stronger than being filtered on the $\L_\infty$-al\-ge\-bras that we will consider.

\begin{definition}\label{def:locally finite}
	A filtered $\L_\infty$-algebra $(\g,\F\g)$ is \emph{locally finite} if all of the quotients $\g/\F_n\g$, $n\ge1$, are finite dimensional.
\end{definition}

Given a filtered $\L_\infty$-algebra $(\g,\F\g)$ and a commutative algebra $A$, from now on we will denote by
\[
\g\widehat{\otimes}A\coloneqq\lim_n\big(\g/\F_n\g\otimes A\big)
\]
the tensor product completed with respect to the filtration on $\g$. Since we will never consider different filtrations on the same algebra, we omit the filtration from the notation.

\begin{lemma}\label{lemma:isomHomTensor}
	Let $(\g,\F\g)$ be a filtered $\L_\infty$-algebra, and let $C$ be a cocommutative coalgebra. Suppose that either
	\begin{enumerate}
		\item the cocommutative coalgebra $C$ is finite dimensional, or
		\item the filtered $\L_\infty$-algebra $(\g,\F\g)$ is locally finite.
	\end{enumerate}
	Then we have an isomorphism
	\[
	\g\widehat{\otimes}C^\vee\cong\hom^\iota(C,\g)
	\]
	of $\L_\infty$-algebras, where $\iota:\com^\vee\to\L_\infty$ is the canonical twisting morphism.
\end{lemma}

\begin{proof}
	The first case is straightforward, so we only give some details for the second one. First begin by considering the case where $\g$ is finite dimensional. If we fix a homogeneous basis $\{x_i\}_i$ of $\g$, then we obtain an isomorphism
	\[
	\hom^\iota(C,\g)\longrightarrow\g\otimes^\iota C^\vee
	\]
	by sending
	\[
	\phi\longmapsto\sum_ix_i\otimes(x_i^\vee\phi)\ .
	\]
	It is a straightforward exercise to check that this is independent of the chosen basis, and to see that the isomorphism holds true at the level of $\L_\infty$-algebras (e.g. using \cite[Lemma 3.8]{rn17}).
	
	\medskip
	
	Now if $\g$ is not necessarily finite dimensional, but only locally finite, we have
	\begin{align*}
		\hom^\iota(C,\g) \cong&\ \hom^\iota(C,\lim_n\g/\F_n\g)\\
		\cong&\ \lim_n\hom^\iota(C,\g/\F_n\g)\\
		\cong&\ \lim_n(\g/\F_n\g\otimes C^\vee)\\
		=&\ \g\widehat{\otimes}C^\vee\ ,
	\end{align*}
	where the fact that the second isomorphism holds at the level of $\L_\infty$-algebras is straightforward to check, and in the third line we used the fact that $\g^{(n)}$ is finite dimensional for all $n$ in order to apply what said above.
\end{proof}

Finally, there is a last condition we will impose on some of our $\L_\infty$-algebras. It was first introduced in \cite{ber15}.

\begin{definition}\label{def:degree-wise nilpotent}
	Let $(\g,\F\g)$ be a filtered $\L_\infty$-algebra. We say that $(\g,\F\g)$ is \emph{degree-wise nilpotent} if for any $n\in\mathbb{Z}$ there is a $k\ge1$ such that $(\F_k\g)_n = 0$.
\end{definition}

The functor $\MC_\bullet$ acts in a very straightforward manner on degree-wise nilpotent filtered $\L_\infty$-algebras satisfying a boundedness condition with respect to the homological degree, as the following result demonstrates.

\begin{proposition}\label{prop:MC with deg-wise nilpotent filtration}
	Let $(\g,\F\g)$ be a degree-wise nilpotent, filtered $\L_\infty$-algebra, and suppose that the degrees in which $\g$ is non-zero are bounded below. Then
	\[
	\MC_\bullet(\g,\F\g)\cong\MC(\g\otimes\Omega_\bullet)\ .
	\]
	In particular, $\MC_\bullet(\g,\F\g)$ is independent of the filtration $\F\g$, as long as $(\g,\F\g)$ is degree-wise nilpotent.
\end{proposition}

\begin{proof}
	Suppose that $(\g,\F_\bullet\g)$ satisfies the assumptions above. Fix $n\ge0$, then there exists $k_0\ge1$ such that $(\F_{k_0}\g)_{m} = 0$ for all $m\le n$, since the degrees of $\g$ are bounded below. It follows that
	\[
	(\F_k\g\otimes\Omega_n)_0 = 0\qquad\text{and}\qquad(\F_k\g\otimes\Omega_n)_{-1} = 0
	\]
	for all $k\ge k_0$, as $\Omega_n$ is concentrated in degrees from $-n$ up to $0$. The projection
	\[
	\g\otimes\Omega_n\longrightarrow\g/\F_k\g\otimes\Omega_n
	\]
	has kernel $\F_k\g\otimes\Omega_n$, and is therefore an isomorphism in degrees $0$ and $-1$. Therefore, the set of Maurer--Cartan elements $\MC(\g/\F_k\g\otimes\Omega_n)$ is constant in $k$ for $k\ge k_0$, which implies the statement.
\end{proof}

\begin{lemma}\label{lemma:sccftGivesDegWiseNil}
	Let $\g$ be a simply-connected, complete $\L_\infty$-algebra. Then $(\g,\F^{can}\g)$ is degree-wise nilpotent.
\end{lemma}

\begin{proof}
	Suppose we take $k$ elements and we bracket them together. We can use at most $k-1$ brackets (taking only binary brackets), and every element has degree at least $2$. It follows that the resulting element has degree at least
	\[
	(1-k) + 2k = k + 1\ .
	\]
	Thus,
	\[
	\F_k^{can}\g\subseteq\g_{\ge k+1}\ ,
	\]
	and the statement follows.
\end{proof}

\subsection{Homotopy Lie algebra and homotopy cocommutative coalgebra models for rational spaces}

In rational homotopy theory, one models spaces via various types of algebras and coalgebras over $\q$. The original approach of Quillen \cite{Quil69} was done using Lie algebras or cocommutative coalgebras. We review here a few possible more general approaches using $\L_\infty$-algebras and homotopy cocommutative coalgebras.

\medskip

First of all, let's look at $\L_\infty$-algebra models, which are a natural generalization of the more classical Lie models.

\begin{definition}\label{def:Loo models}
	Let $K$ be a simplicial set. A filtered $\L_\infty$-algebra $(\g,\F\g)$ is a \emph{rational model} for $K$ if there is a homotopy equivalence
	\[
	\MC_\bullet(\g,\F\g)\simeq K_\q\ ,
	\]
	where $K_\q$ is the rationalization of $K$, see \cite[Chap. 9]{FHT01} for more details.
\end{definition}

From now on, we will require that our $\L_\infty$-models are locally finite and de\-gree-wise nilpotent. This assumption is needed e.g. for \cref{thm:Berglund} to hold. This assumption is not that strong, as we can model a great number of spaces with such models. Recall that a simplicial complex is said to be $1$-reduced if it has a single $0$-cell and no $1$-cells.

\begin{proposition}\label{prop:condition ok}
	Every $1$-reduced finite simplicial complex $K$ admits a locally finite, degree-wise nilpotent $\L_\infty$-model.
\end{proposition}

\begin{proof}
	According to \cite[Sect. 24.(e)]{FHT01}, every simplicial complex $K$ of finite type admits a free Lie model --- in the sense of the Definition found at \cite[p.322]{FHT01} --- of the form $\g=\lie(C_\bullet(X))$, where $C_\bullet(X)$ is the complex of simplicial chains of $K$. Since we supposed that $K$ is finite, $\g$ is a finitely generated Lie algebra, and thus it satisfies the assumptions of \cite[Prop. 6.1]{ber15}. It follows that we have
	\[
	\MC_\bullet(\g\otimes\Omega_\bullet)\simeq K_\q\ .
	\]
	An apparent problem is the fact that $\g$ is not complete. However, we can replace $\g$ by
	\[
	\widehat{\g}\coloneqq\widehat{\lie}(C_\bullet(X)) = \prod_{n\ge1}\lie(n)\otimes_{\S_n}C_\bullet(X)^{\otimes n}\ .
	\]
	Then \cref{prop:MC with deg-wise nilpotent filtration} and \cref{lemma:sccftGivesDegWiseNil} give
	\[
	\MC_\bullet(\g\otimes\Omega_\bullet)\cong\MC_\bullet(\widehat{\g},\F^{can}\widehat{\g})\ ,
	\]
	so that we may use $(\widehat{\g},\F^{can}\widehat{\g})$ as a $\L_\infty$-model. This is obviously locally finite, since the operad $\lie$ is finite dimensional in every arity and $C_\bullet(X)$ is finite dimensional, and degree-wise nilpotent e.g. by \cref{lemma:sccftGivesDegWiseNil}.
\end{proof}

Alternatively to using Lie or $\L_\infty$-models in order to do rational ho\-mo\-to\-py theory, one can use cocommutative coalgebras or commutative algebras.

\begin{theorem}[{\cite[Thm. 1]{Quil69}}]
	There is a functor
	\[
	\quil:\Top_{*,1}\longrightarrow\cocom_{\ge2}
	\]
	from the category of simply-connected, pointed topological spaces to the category of sim\-ply-con\-nected cocommutative coalgebras which induces an equivalence between the res\-pec\-tive homotopy categories.
\end{theorem}

This functor is constructed in two steps, first Quillen constructs a functor $\lambda$ from spaces to differential graded Lie algebras. Then he applies the bar construction relative to the canonical twisting morphism $\kappa:\com^{\vee} \rightarrow \lie$ to get a cocommutative model for the space $X$.

\medskip 

Dually, one has the following approach with commutative algebras, see \cite{Sul77} and \cite[Chap. 17]{FHT01}.

\begin{theorem}
	There is a functor
	\[
	A_{PL}^*:\widetilde{\Top}_{*,1}^{op}\longrightarrow\{\text{Com. alg. with homology concentrated in degree }\le-2\}
	\]
	from the (opposite category of the) category of simply-connected, pointed topological spaces of finite $\q$-type to the category of  commutative algebras over $\q$, such that the homology of these algebras is concentrated in degree less than or equal to $-2$. This functor induces an equivalence between the respective homotopy categories.
\end{theorem}

\begin{remark}
Recall that we are working with a homological grading, so in particular $A_{PL}^{*}$ will be negatively graded.
\end{remark}

We put ourselves in a slightly more general context, and use homotopy cocommutative coalgebras instead of strict cocommutative coalgebras.

\begin{definition}
	Let $X$ be a simply-connected, pointed topological space. A homotopy cocommutative coalgebra $C$ is a \emph{rational model} for $X$ if there exists a zig-zag of weak equivalences of homotopy cocommutative coalgebras
	\[
	(f_\iota)_*\quil(X)\longrightarrow\bullet\longleftarrow\cdots\longrightarrow\bullet\longleftarrow C\ ,
	\]
	or equivalently (by \cref{prop:alpha-we and zig-zags}) if there exists a $\pi$-weak equivalence $(f_\iota)_*\quil(X)\rightsquigarrow C$.
\end{definition}

The version using algebras is as follows.

\begin{definition}
	Let $X$ be a simply-connected, pointed topological space of finite $\q$-type. A commutative algebra $A$ is a \emph{rational model} for $X$ if there exists a zig-zag of quasi-isomorphisms of commutative algebras
	\[
	A\longleftarrow\bullet\longrightarrow\cdots\longleftarrow\bullet\longrightarrow A_{PL}^*(X)\ ,
	\]
	or equivalently (by \cite[Thm. 11.4.9]{lodayvallette}) if there exists an $\infty_{\iota}$-quasi-isomorphism $A\rightsquigarrow A^*_{PL}(X)$ of $\C_\infty$-algebras.
\end{definition}

For any given $X\in\widetilde{\Top}_{*,1}$ there exists such a model. For example, one can trivially take $A_{PL}^*(X)$ as commutative rational model for $X$.

\medskip

Commutative and cocommutative models are strictly related, as one would expect. In order to prove it, we recall the following theorem, due to Majewski.

\begin{theorem}[{\cite[Thm. 4.90]{majewski}}]\label{thm:majewski}
	Let $X$ be a simply-connected space. Let
	\[
	\ell_X:M_X\stackrel{\sim}{\longrightarrow}A_{PL}^*(X)
	\]
	be a simply-connected commutative model of finite type for $X$, and let
	\[
	\mu_X:L_X\stackrel{\sim}{\longrightarrow}\lambda(X)
	\]
	be a free suspended Lie model for $X$. There exists a canonical homotopy class of quasi-isomorphisms
	\[
	\alpha_X:L_X\stackrel{\sim}{\longrightarrow}\Omega_\kappa(M_X^\vee)\ .
	\]
\end{theorem}

\begin{theorem}\label{thm:dualization of models}
	Let $X$ be a simply-connected space of finite $\q$-type.
	\begin{enumerate}
		\item \label{1} Let $A$ be a commutative model of finite type for $X$. Then its dual $A^\vee$ is a cocommutative model for $X$.
		\item \label{2} Dually, let $C$ be a cocommutative model for $X$. Then its linear dual $C^\vee$ is a commutative model for $X$.
	\end{enumerate}
	Notice that we do not have any finiteness assumption on our cocommutative models.
\end{theorem}

\begin{proof}
	We begin by proving (\ref{1}). Let $A$ be a simply-connected commutative model of finite type for $X$. It is well known that every simply-connected space $X$ of finite $\q$-type admits a minimal commutative model $M_X$, which in particular is simply-connected and of finite type, see e.g. \cite[p.146]{FHT01}. It follows that we have a zig-zag of quasi-isomorphisms
	\[
	A\longleftarrow\bullet\longrightarrow M_X
	\]
	by \cite[Thm. 11.4.9]{lodayvallette}. Inspecting the proof in loc. cit. we notice that we can take $\Omega_\iota\bar_\iota M_X$ as intermediate algebra, which is again simply-connected and of finite type. Dualizing linearly, we obtain a zig-zag of quasi-isomorphisms
	\begin{equation}\label{eq:zigzag}
		A^\vee\longrightarrow\bullet\longleftarrow M_X^\vee\ ,
	\end{equation}
	where all terms are well-defined coalgebras thanks to the fact that they are of finite type. Finally, we obtain a zig-zag
	\[
	A^\vee\longrightarrow\bar_\pi\Omega_\pi A^\vee\longrightarrow\bar_\pi\Omega_\pi(\bullet)\longleftarrow\bar_\pi\Omega_\pi M_X^\vee\longleftarrow\bar_\kappa\lambda(X)\ ,
	\]
	where the first arrow is the unit of the bar-cobar adjunction and is a weak equivalence by \cite[Thm. 2.6(2)]{val14}, the second and third arrows are obtained by the arrows of the zig-zag (\ref{eq:zigzag}) by applying $\bar_\pi\Omega_\pi$ and are also weak equivalences. The last arrow is obtained as follows. The Lie algebra $\lambda(X)$ is a Lie model for $X$. Therefore, by \cref{thm:majewski} there is a quasi-isomorphism
	\[
	\alpha_X:\lambda(X)\stackrel{\sim}{\longrightarrow}\Omega_\pi(M_X^\vee)\ .
	\]
	Applying the bar construction we obtain the desired weak equivalence
	\[
	\bar_\pi\alpha_X:\bar_\pi\lambda(X)\longrightarrow\bar_\pi\Omega_\pi(M_X^\vee)\ ,
	\]
	concluding the first part of the proof.
	
	\medskip
	
	For point (\ref{2}), let $C$ be a cocommutative model for $X$. By point (\ref{1}), we have in particular that the dual $M_X^\vee$ of the minimal commutative model $M_X$ is a cocommutative model for $X$. Therefore, we have a zig-zag of weak equivalences
	\[
	C\longrightarrow\bullet\longleftarrow M_X^\vee\ ,
	\]
	which in particular are quasi-isomorphisms. Dualizing linearly, we obtain a zig-zag of quasi-isomor\-phisms
	\[
	C^\vee\longleftarrow\bullet\longrightarrow M_X^{\vee\vee}\cong M_X\ ,
	\]
	where the last isomorphism holds because $M_X$ is of finite type. Therefore, the commutative algebra $C^\vee$ is a commutative model for $X$.
\end{proof}

\subsection{Rational models for mapping spaces}

Given two spaces $K$ and $L$, a natural question is the following one. Suppose we are given rational models for both $K$ and $L$. Is it possible to use them to construct a rational model of the based mapping space $\map(K,L)$? A possible answer to this question was given by Berglund \cite{ber15} in the case when we have a strictly commutative model for the first space, and an $\L_\infty$-model for the second one.

\begin{theorem}[\cite{ber15} Theorem 6.3]\label{thm:Berglund}
	Let $K$ be a simply-connected simplicial set, let $L$ be a nilpotent space (e.g. a simply-connected space) of finite $\q$-type and $L_{\q}$ the $\q$-localization of $L$. Let $A$  be a commutative model for $K$ and $(\g,\F\g)$ a degree-wise nilpotent, locally finite $\L_{\infty}$-model of finite $\q$-type for $L$. There is a homotopy equivalence of simplicial sets
	\[
	\map(K,L_{\q}) \simeq \MC_{\bullet} (A \widehat{\otimes }\g) ,
	\]
	i.e. the $\L_\infty$-algebra $A \widehat{\otimes }\g$ is an $\L_\infty$-model for the mapping space.
\end{theorem}

\begin{remark}
	In \cite{ber15}, this theorem is stated in terms of the Getzler $\infty$-groupoid $\gamma_{\bullet}$. However, the $\infty$-groupoid $\gamma_{\bullet}$ is homotopy equivalent to $\MC_{\bullet}$ by \cite[Cor. 5.9]{Get09}, and thus the statement above is equivalent to the original one. Also notice that we supposed that $(\g,\F\g)$ is locally finite, and completed the tensor product with respect to the filtration $\F\g$ and not with the degree filtration, as in \cite{ber15}. An inspection of the original proof reveals that the result still holds in this slightly more general context.
\end{remark}

Using the results of the present article, we will improve Berglund's Theorem in two ways: we will show that we can take homotopy cocommutative coalgebra models for $K$ instead of just cocommutative ones, and that this model is natural with respect to $\infty_\pi$-quasi-isomorphisms of  $\L_\infty$-algebras, respectively $\pi$-weak equivalences of homotopy cocommutative coalgebras. We will also show that, under certain restrictions on $C$ and $\g$, this model only depends on the homotopy types of $C$ and $\g$, i.e. different choices for $C$ and $\g$ will give homotopy equivalent models for the mapping space. The first result is the following one.

\begin{lemma}
	Let $K$ be a simply-connected simplicial set, let $L$ be a nilpotent space (e.g. a simply-connected space) of finite $\q$-type and $L_{\q}$ the $\q$-localization of $L$. Let $C$  be a cocommutative model for $K$ and $(\g,\F\g)$ a degree-wise nilpotent, locally finite $\L_{\infty}$-model of finite type for $L$. There is a homotopy equivalence of simplicial sets
	\[
	\map(K,L_{\q}) \simeq \MC_\bullet(\hom^\pi((f_\iota)_*C,\g)) ,
	\]
	i.e. the convolution $\L_\infty$-algebra $\hom^\pi((f_\iota)_*C,\g)$ is an $\L_\infty$-model for the mapping space.
\end{lemma}

\begin{proof}
	By \cref{thm:dualization of models} and \cref{thm:Berglund}, we know that $\g\otimes C^\vee$ is an $\L_\infty$-model for the mapping space. Further, by \cref{lemma:isomHomTensor} we know that
	\[
	\g\widehat{\otimes}C^\vee\cong\hom^\iota(C,\g) = \hom^\pi((f_\iota)_*C,\g)\ ,
	\]
	where the second equality is \cref{lemma:naturality of convolution algebras wrt compositions}.
\end{proof}

\begin{proposition}
	Let $C$ be a homotopy cocommutative coalgebra, and let
	\[
	\Psi:(\h,\F\h)\rightsquigarrow(\g,\F\g)\ .
	\]
	be a filtered $\infty$-morphism of $\L_\infty$-algebras. Then there is a weak equivalence of simplicial sets
	\[
	\MC_\bullet(\hom^\pi(C,\h))\simeq\MC_\bullet(\hom^\pi(C,\g))\ .
	\]
\end{proposition}

\begin{proof}
	This looks very similar to \cref{thm:invariance of MC spaces} --- and indeed, that result is a fundamental ingredient of the proof --- but notice that here we have $\infty$-morphisms of $\L_\infty$-algebras, i.e. $\infty_\iota$-morphisms, instead of $\infty_\pi$-morphisms.
	
	\medskip
	
	The proof is schematized by the following diagram.
	\begin{center}
		\begin{tikzpicture}
			\node (a) at (0,5){$\hom^\pi(C,\h)$};
			\node (b) at (0,2.5){$\hom^\pi(R_{\pi,f_\iota}(C),\h)$};
			\node (c) at (0,0){$\hom^\iota(\bar_\iota\Omega_\pi C,\h)$};
			\node (d) at (6,0){$\hom^\iota(\bar_\iota\Omega_\pi C,\g)$};
			\node (e) at (6,2.5){$\hom^\pi(R_{\pi,f_\iota}(C),\g)$};
			\node (f) at (6,5){$\hom^\pi(C,\g)$};
			
			\draw[->,line join=round,decorate,decoration={zigzag,segment length=4,amplitude=.9,post=lineto,post length=2pt}] (b) -- node[left]{$\hom^\pi(E_D,1)$} (a);
			\draw[double equal sign distance] (b) -- (c);
			\draw[->,line join=round,decorate,decoration={zigzag,segment length=4,amplitude=.9,post=lineto,post length=2pt}] (c) -- node[above]{$\hom^\iota(1,\Psi)$} (d);
			\draw[double equal sign distance] (d) -- (e);
			\draw[->,line join=round,decorate,decoration={zigzag,segment length=4,amplitude=.9,post=lineto,post length=2pt}] (e) -- node[left]{$\hom^\pi(E_D,1)$} (f);
		\end{tikzpicture}
	\end{center}
	The vertical equalities are given by \cref{lemma:naturality of convolution algebras wrt compositions} and the definition of the rectification $R_{\pi,f_\iota}$. We apply the functor $\MC_\bullet$ on the whole diagram, and all the squiggly arrows become homotopy equivalences of simplicial sets by \cref{thm:invariance of MC spaces} and \cref{prop:rectification is oo-qi to original}. The result follows.
\end{proof}

Our generalization of \cref{thm:Berglund} is a direct consequence of this result.

\begin{theorem}\label{cor:generalization of Berglund}
	Let $K$ be a simply-connected simplicial set, let $L$ be a nilpotent space (e.g. a simply-connected space) of finite $\q$-type and $L_{\q}$ the $\q$-localization of $L$. Let $C$  be a homotopy cocommutative model for $K$, let $(\g,\F\g)$ be a degree-wise nilpotent, locally finite $\L_{\infty}$-model of finite type for $L$, and let $(\h,\F\h)$ be an $\L_\infty$-algebra such that there exists a filtered $\infty$-quasi-isomorphism
	\[
	\Psi:(\h,\F\h)\rightsquigarrow(\g,\F\g)\ .
	\]
	 There is a homotopy equivalence of simplicial sets
	\[
	\map(K,L_{\q}) \simeq \MC_\bullet(\hom^\pi(C,\h)) ,
	\]
	i.e. the convolution $\L_\infty$-algebra $\hom^\pi(C,\h)$ is an $\L_\infty$-model for the mapping space.
\end{theorem}

An example of an application of this theorem is an alternative proof of \cite[Thm. 3.2]{BG16}, see \cite[Cor. 11.1]{w16}.

\subsection{Another application: Algebraic Hopf invariants and the mo\-du\-li space of Maurer--Cartan elements}

Another application of the theory developed in the present paper is given in \cite{w17}, where the second author uses it to construct a complete invariant of the real or rational homotopy classes of maps between simply-connected manifolds.

\medskip

Let $M$ and $N$ be two smooth, orientable and simply-connected based manifolds, and suppose that $M$ is compact. The homology $H_\bullet(M)$ of $M$ has a $\C_\infty$-coalgebra structure making it into a model for $M$, and similarly, the rational homotopy $\pi_\bullet(N)\otimes\q$  of $N$ has a $\L_\infty$-algebra structure making it into a model for $N$. By \cref{cor:generalization of Berglund}, the convolution $\L_\infty$-algebra $\hom(H_\bullet(M),\pi_\bullet(N)\otimes\q)$  becomes a model for the based mapping space $\map(M,N)$. In \cite{w16}, the second author constructed a map
\[
\mathrm{mc}:\map(M,N)\longrightarrow\MC(\hom(H_\bullet(M),\pi_\bullet(N)\otimes\q))
\]
from the based mapping space $\map(M,N)$ to the set of Maurer--Cartan elements of the convolution algebra. This map is not a homotopy invariant of maps, but by \cite[Theorem 7.1]{w16}, two maps $f,g:M\to N$ are homotopic if, and only if $\mathrm{mc}(f)$ and $\mathrm{mc}(g)$ are gauge equivalent. Therefore, one obtains a complete homotopy invariant
\[
\mathrm{mc}_\infty:\map(M,N)\longrightarrow\mathcal{MC}(M,N)
\]
of maps from $M$ to $N$ by composing $\mathrm{mc}$ with the projection onto the moduli space
\[
\mathcal{MC}(M,N)\coloneqq\MC(\hom(H_\bullet(M),\pi_\bullet(N)\otimes\q))/\sim_{\mathrm{gauge}}
\]
of Maurer--Cartan elements. In \cite{w17}, the second author uses the techniques developed in this paper to study the moduli space $\mathcal{MC}(M,N)$ with the goal of computing this complete invariant of maps. More precisely, this is done as follows. First one fixes a minimal CW-complex $X$, rationally equivalent to $M$. Here, a CW-complex is called minimal if its associated cellular chain complex has zero differential. If we denote the $n$-skeleton of $X$ by $X_n$ and the attaching maps by
\[
a_{n+1} : \bigvee_{k_{n+1}} S^n \rightarrow X_n\ ,
\]
where $k_{n+1}$ is the number of $(n+1)$-cells. We denote the inclusion map of the $n$-skeleton into the $(n+1)$-skeleton by $i_n:X_n \rightarrow X_{n+1}$. Using this minimal CW-decomposition, one gets a tower of fibrations approximating the moduli space of Maurer--Cartan elements: each inclusion map $i_n$ induces a fibration of simplicial sets
\[
i_n^*:\MC(\hom(X_{n+1},\pi_\bullet(N)\otimes\q))\rightarrow \MC(\hom(X_n,\pi_\bullet(N)\otimes\q))\ .
\] 
This way, one obtains a tower of fibrations approximating the moduli space of Maurer--Cartan elements. Thanks to the results of the present paper, the second author is able to give a description of the fibers of the fibrations in this tower in terms of the attaching maps $a_n$. Using this techniques one can often give very concrete descriptions of the approximations of the moduli space of Maurer--Cartan elements.

\bibliographystyle{alpha}
\bibliography{Homotopy_morphisms_between_convolution_homotopy_Lie_algebras}

\end{document}